\documentclass[a4paper,10pt,reqno]{amsart}
\usepackage[utf8]{inputenc}

\usepackage{amssymb}
\usepackage{amsmath}

\usepackage{color}

\usepackage[colorlinks=true]{hyperref}

\definecolor{dark-red}{rgb}{.54,.0,.0}
\definecolor{dark-green}{rgb}{.0,.4,.0}
\definecolor{dark-blue}{rgb}{.04,.04,.4}

\hypersetup{linkcolor=dark-red, urlcolor=dark-blue, citecolor=dark-green}

\def\Rb{{\mathbb R}}

\newcommand{\nc}{\newcommand}
\newcommand{\fhi}{\varsigma}
\renewcommand{\phi}{\varphi}
\newcommand{\eps}{\varepsilon}
\nc{\txt}{\textstyle}
\nc{\be}{\begin{equation}}
\nc{\ee}{\end{equation}}
\nc{\ba}{\begin{eqnarray}}
\nc{\ea}{\end{eqnarray}}
\nc{\bas}{\begin{eqnarray*}}
\nc{\eas}{\end{eqnarray*}}
\nc{\weak}{\rightharpoonup}
\nc{\paq}{$({\mathcal{P}})_{\alpha,q}$}
\nc{\paz}{$({\mathcal{P}})_{\alpha,\tz}$}
\nc{\as}{s}
\nc{\bt}{t}
\nc{\tf}{{2^\flat}}
\nc{\Om}{\Omega}
\nc{\ek}{{\eps_{k}}}
\nc{\ak}{{\alpha_{k}}}
\nc{\ck}{{C_{k}}}
\nc{\ct}{\tilde{C}}
\nc{\cc}{\hat{C}}
\nc{\uk}{{u_{k}}}
\nc{\Uk}{{U_{k}}}
\nc{\ts}{{2^*}}
\nc{\tsu}{{{2^*}-1}}
\nc{\tsd}{{{2^*}-2}}
\nc{\tst}{{{2^*}-3}}
\nc{\tz}{{2^\#}}
\nc{\tzu}{{{2^\#}-1}}
\nc{\tzd}{{{2^\#}-2}}
\nc{\wk}{{w_{k}}}
\nc{\sndd}{\frac{S^\frac N2}2}
\nc{\sddn}{\frac{S}{2^{\frac 2N}}}
\nc{\igukw}{{\int\nabla\Uk\cdot\nabla\wk}}
\nc{\iukw}{{\int U_k\wk}}
\nc{\iuks}{{|U_k|_{2^*}^{2^*}}}
\nc{\iukz}{{|U_k|_{2^\#}^{2^\#}}}
\nc{\iukpz}{{|u_k|_{2^\#}^{2^\#}}}
\nc{\iusu}{{\int U_k^{{2^*}-1} w_k}}
\nc{\iusd}{{\int U_k^{{2^*}-2} w_k^2}}
\nc{\iust}{{\int U_k^{{2^*}-3} w_k^3}}
\nc{\iuzu}{{\int U_k^{{2^\#}-1} |w_k|}}
\nc{\iuzd}{{\int U_k^{{2^\#}-2} w_k^2}}
\nc{\iuz}{{|u|_{\tz}^\tz}}
\nc{\ius}{{|u|_{\ts}^\ts}}
\nc{\nguk}{{|\nabla\Uk|_{2}}}
\nc{\ngukq}{{|\nabla\Uk|_{2}^2}}
\nc{\ngwkq}{{|\nabla\wk|_{2}^2}}
\nc{\nuk}{{|\Uk|_{2}}}
\nc{\nukq}{{|\Uk|_{2}^2}}
\nc{\nuks}{{|\Uk|_{\ts}}}
\nc{\nuksq}{{|\Uk|_{\ts}^2}}
\nc{\nukss}{{|\Uk|_{\ts}^\ts}}
\nc{\nuksmd}{{|\Uk|_{\ts}^{-2}}}
\nc{\nwk}{{||\wk||}}
\nc{\nwks}{{|\wk|_{\ts}}}
\nc{\nwksq}{{|\wk|_{\ts}^2}}
\nc{\ndwkq}{{|\wk|_{2}^2}}
\nc{\nwkq}{{||\wk||^2}}
\nc{\nwkr}{{||\wk||^r}}
\nc{\ngukpq}{{|\nabla u_{k}|_{2}^2}}
\nc{\ngupq}{{|\nabla u|_{2}^2}}
\nc{\nukps}{{|u_{k}|_{\ts}}}
\nc{\nukpss}{{|u_{k}|_{\ts}^\ts}}
\nc{\nups}{{|u|_{\ts}}}
\nc{\nupss}{{|u|_{\ts}^\ts}} 
\nc{\nupsz}{{|u|_{\ts}^\tz}} 
\nc{\nukpsq}{{|u_{k}|_{\ts}^2}}
\nc{\nupsq}{{|u|_{\ts}^2}}
\nc{\nupq}{{|u|_{2}^2}}
\nc{\nukpq}{{|u_k|_{2}^2}}
\nc{\ql}{\frac{\ngukq}{|\Uk|_{\ts}^{2^*}}}
\nc{\beuk}{\frac{\ngukq+a\nukq}{\nuksq}}
\nc{\nhu}{||u||^2}
\nc{\pnhu}{||u||^2}
\nc{\pnhumeio}{||u||}
\nc{\nhukpq}{||u_{k}||^2}
\nc{\pnhukpq}{||u_{k}||^2}
\nc{\pnhukpqmeio}{||u_{k}||}
\nc{\nhukq}{||U_{k}||^2}
\nc{\pnhukq}{||U_{k}||^2}
\nc{\nhwkq}{||w_{k}||^2}
\nc{\pnhwkq}{||w_{k}||^2}
\nc{\ttt}{{\Rb^{N}_{+}}}
\nc{\ue}{U_{\eps}}
\nc{\pb}{\bar{\phi}_{\eps}}
\nc{\pbq}{\bar{\phi}_{\eps}^2}
\nc{\pk}{{\phi}_{\eps}}
\nc{\uz}{U^\tzd_{\eps}}
\nc{\uzz}{U^\tzd_{\eps,y_{\eps}}}
\nc{\us}{U^\tsd_{\eps}}
\nc{\uss}{U^\tsd_{\eps,y_{\eps}}}
\nc{\et}{\eta_{\eps}}
\nc{\etq}{\eta_{\eps}^2}
\nc{\pt}{{\tilde\phi}}
\nc{\pte}{{\tilde{\phi}_{\eps}}}
\nc{\ptq}{{{\tilde\phi}^2}}
\nc{\pteq}{\tilde{\phi}_{\eps}^2}
\nc{\ome}{{\Om_{\eps}}}
\nc{\mk}{M_{k}}
\nc{\dk}{\delta_{k}}
\nc{\pkk}{P_{k}}
\nc{\ndd}{{\frac{N-2}{2}}}
\nc{\ium}{U_{\delta_{k},P_{k}}}
\nc{\vk}{{v_{k}}}
\nc{\gb}{\gamma+\sqrt{\gamma^2+4\beta}}
\nc{\gbb}{\gamma_{\infty}+\sqrt{\gamma_{\infty}^2+4\beta_{\infty}}}
\nc{\ia}{\frac{1}{4(\tz)^{\frac{2}{N}}}
         \left[\left(\gb\right)^{N}+2\cdot\ts\beta\left(\gb\right)^{N-2}
         \right]^{\frac{2}{N}}}
\nc{\iab}{\frac{1}{4(\tz)^{\frac{2}{N}}}
         \left[\left(\gbb\right)^{N}+2\cdot\ts\beta_{\infty}\left(\gbb\right)^{N-2}
         \right]^{\frac{2}{N}}}
\nc{\aaa}{\mbox{{\sl a}}}
\nc{\bg}{\mbox{\b{$\gamma$}}}
\nc{\bb}{\mbox{\b{$\beta$}}}
\nc{\zs}{{\frac{\tz}{\ts}}}
\nc{\sz}{{\frac{\ts}{\tz}}}
\nc{\ds}{{\frac{2}{\ts}}}
\nc{\gx}{\bg x^{\zs}}
\nc{\sd}{S_{2}}
\nc{\gxq}{\bg^2 x^{2\zs}}
\nc{\az}{S^{\frac N2}\frac{A(N)}{B(N)}\max_{\partial\Om}H}
\nc{\azz}{A(N)\max_{\partial\Om}H}
\nc{\bd}{\mbox{\b{$\delta$}}}
\nc{\dd}{\frac{1}{2}\alpha\delta}
\nc{\tzud}{\frac{\tz+1}{2}}
\nc{\tzdois}{\frac{\tz}{2}}
\nc{\lb}{\left(}
\nc{\rb}{\right)}
\nc{\defdel}{\frac{\iuz}{\pnhumeio\cdot|u|_{\ts}^{\ts\!/2}}}
\nc{\defdeli}{\frac{|u|_{q}^{q\bt}}{\pnhumeio^{2-\as}
     |u|_{\ts}^{\ts\as/2}}}
\nc{\intum}{\int(\nabla u\cdot\nabla\phi+au\phi)}
\nc{\intdois}{\int u^\tzu\phi}
\nc{\inttres}{\int u^{\tsu}\phi}
\nc{\cs}{{\mathcal{S}}}
\nc{\sddnd}{S/2^{\frac{2}{N}}}   
\newcounter{um}
\newcounter{dois}

\nc{\uq}{|u|_{q}^q}
\nc{\uqb}{|u|_{q}^{q\bt}}
\nc{\qa}{\lb\frac{\pnhumeio}{|u|_{\ts}^{\ts\!/2}}\rb^\as}
\nc{\qau}{\frac{\pnhumeio^\as}{|u|_{\ts}^{2+\ts\as/2}}}
\nc{\beps}{\epsilon}
\nc{\uqz}{U^{q\bt-2}_{\eps,y_{\eps}}}
\nc{\gxi}{\bb x^\ds+{\textstyle
    \frac{||\mu||}{||\nu||^{2/\ts}}}(1-x)^\ds}
\nc{\gxil}{\bb x+{\textstyle
    \frac{||\mu||}{||\nu||^{2/\ts}}}(1-x)}
\nc{\du}{A_{1}}
\nc{\ddois}{A_{2,k}}
\nc{\dt}{A_{3,k}}
\nc{\fu}{B_{1,k}}
\nc{\fd}{B_{2,k}}
\nc{\ft}{B_{3,k}}

\newtheorem{thm}{Theorem}[section]

\newtheorem{lem}[thm]{Lemma}

\newtheorem{rem}[thm]{Remark}

\begin{document}

\title{A family of sharp inequalities for Sobolev functions}

\author{Pedro M.\ Gir\~{a}o}
\thanks{Partially supported
    by FCT (Portugal).}
\address{Mathematics Department, Instituto Superior T\'{e}cnico, Av.\ 
    Rovisco Pais, 1049-001 Lisbon, Portugal}
\email{pgirao@math.ist.utl.pt}

   \begin{abstract}
         \noindent 
         Let $N\geq 5$, $\Om$ be a smooth bounded domain in 
         $\Rb^{N}$, 
         $\ts=\frac{2N}{N-2}$, $a>0$,
         $S=\inf\left\{\left.
         \int_{\Rb^{N}}|\nabla u|^2\,\right|\,
         u\in L^\ts(\Rb^{N}), \nabla u\in L^2(\Rb^{N}),
         \int_{\Rb^{N}}|u|^\ts=1
         \right\}$
         and 
         $||u||^2=
         |\nabla u|_{2}^2+a|u|_{2}^2$. 
         We define
         $\tf=\frac{2N}{N-1}$,
$\tz=\frac{2(N-1)}{N-2}$
and consider $q$ such that $\tf\leq q\leq\tz$.  We also define
$\as=2-N+\frac{q}{\ts-q}$
and
$\bt=\frac{2}{N-2}\cdot\frac{1}{\ts-q}$.
We prove that there exists an $\alpha_{0}(q,a,\Om)>0$ such that, for
all 
$u\in H^1(\Om)\setminus\{0\}$,
$$\sddn\nupsq\leq\nhu+\alpha_{0}\qa\uqb,\eqno{(I)_{q}}$$
where the norms are over $\Om$.
Inequality $(I)_{\tf}$ is due to M.\ Zhu.  
    \end{abstract}
    
\maketitle

\section{Introduction}

Let $N\geq 5$, $\Om$ be a smooth bounded domain in 
$\Rb^{N}$, $\tf=\frac{2N}{N-1}$,
$\tz=\frac{2(N-1)}{N-2}$, $\ts=\frac{2N}{N-2}$, $a>0$
and $||u||^2=|\nabla u|_{2}^2+a|u|_{2}^2$.
Unless otherwise indicated, norms are over $\Om$.
We recall that the 
infimum
$$
S:=\inf_{\stackrel{
\mbox{\tiny $u\in
L^\ts(\Rb^{N})\setminus\{0\}$}} 
{\mbox{\tiny $\nabla u\in L^2(\Rb^{N})$}
}}
\frac{\int_{\Rb^{N}}|\nabla 
u|^2}{\lb\int_{\Rb^{N}}|u|^\ts\rb^{2/\ts}}
$$
is achieved by the Talenti instanton
$
U(x):=\left(\frac{N(N-2)}{N(N-2)+|x|^2}\right)^{\frac{N-2}{2}}
$.

M.\ Zhu proved in \cite{Zh} that there exists $\bar\alpha_{0}
>0$ such that
\be\label{zhu}
\sddn\nupsq\leq||u||^2+\bar\alpha_{0}|u|_{\tf}^2,
\ee
for all $u\in H^1(\Om)$.
It was announced by the author in \cite{Girao} that there exists 
$\tilde\alpha_{0}>0$ such that
\be
\label{pg}
\sddn\nupsq\leq||u||^2+\tilde\alpha_{0}\frac{||u||}{|u|_{\ts}^{\ts/2}}
|u|_{\tz}^{\tz},
\ee
for all $u\in H^1(\Om)\setminus\{0\}$.
In this work we prove a family of inequalities which includes
(\ref{zhu}) and (\ref{pg}) as special cases.

The work of M.\ Zhu was motivated by the works \cite{AM} and
\cite{WX},
by Adimurthi and Mancini and by X.J.\ Wang, respectively. They imply 
that one cannot expect the existence of a constant $\bar\alpha_{0}$ 
such that
$$
\sddn\nupsq\leq||u||^2+\bar\alpha_{0}|u|_{2}^2,
$$
for all $u\in H^1(\Om)$.
In \cite{Zh}, M.\ Zhu raises the $L^2$ norm on the right
hand side to a higher $L^q$ norm in order to obtain an inequality 
valid in $H^1(\Om)$.

The work \cite{Girao} was motivated by \cite{WX}, the referred work of
X.J.\ Wang, and by \cite{G}, by D.G.\ Costa and the author.  
Both \cite{WX} and \cite{G} consider 
the problem
$$
\left\{\begin{array}{ll}
-\Delta u+au+\alpha u^{q-1}=u^\tsu&\mbox{in\ }\Om,\\
u>0&\mbox{in\ }\Om,\\
\frac{\partial u}{\partial\nu}=0&\mbox{on\ }\partial\Om.
\end{array}\right.\eqno{({\mathcal{P}})_{\alpha,q}}
$$
From \cite{WX} we know that if
$q<\tz$, then problem \paq\ has a ground state solution
for all values of $\alpha\geq 0$.  From \cite{G} we know that
there exists 
$\alpha_{0}>0$ such that if $\alpha<\alpha_{0}$, then problem \paz\
    has a ground state solution and
    if  $\alpha>\alpha_{0}$, then problem \paz\ has no ground state 
    solution.  The solutions of \paq\ correspond to 
critical points of 
the functional $\Phi_{\alpha}:H^1(\Om)\to\Rb$, defined by
\be\label{phivelho}
\Phi_{\alpha}(u):=\frac 
12||u||^2+\frac{\alpha}{q}|u|_{q}^q-\frac 1\ts\ius.
\ee
We recall that a 
ground state solution, or least energy solution, of \paq\ is a
function 
$u\in H^1(\Om)$ such that
$$
\Phi_{\alpha}(u)=\inf_{{\mathcal{N}}}\Phi_{\alpha}.
$$
The set ${\mathcal{N}}$ is the Nehari manifold,
${\mathcal{N}}:=\{u\in H^1(\Om)\setminus\{0\}:
\Phi_{\alpha}^\prime(u)u=0\}$.
When $q=\tz$ it is possible to determine 
explicitly the function $\Phi_{\alpha}|_{{\mathcal{N}}}$ by 
solving a quadratic equation. 
The analysis of \cite{G} takes advantage of
this fact.  As a by-product it implies a certain inequality
(see (15) of \cite{G}).  Inequality (\ref{pg}) is an improvement of
the 
inequality in \cite{G}. 

The idea of the proof of inequalities (\ref{zhu}) and (\ref{pg}) is 
 based on an argument by contradiction.  Indeed, consider the 
the functionals $\Psi_{\alpha}:H^1\setminus\{0\}\to\Rb$ defined by 
$$
\Psi_{\alpha}(u)=
\frac{||u||^2}{\nupsq}+\alpha\frac{|u|_{\tf}^2}{\nupsq}
\quad\mbox{or}\quad
\Psi_{\alpha}(u) =
\frac{||u||^2}{\nupsq}+\alpha\frac{||u||}{|u|_{\ts}^{2+\ts/2}}
|u|_{\tz}^{\tz}.
$$
Let $(\ak)$ be any sequence of nonnegative real numbers such that
$\ak\to+\infty$.
If (\ref{zhu}) (respectively (\ref{pg})) is false, then, for each $k$,
$\inf_{H^1(\Om)\setminus\{0\}}
\Psi_{\alpha_{k}}<\sddn$.  This implies that 
$\Psi_{\alpha_{k}}$ has a line of minima
(with 0 removed), 
which are called least energy critical points of $\Psi_{\alpha_{k}}$.
One of these,
$u_{k}$, 
satisfying an appropriate normalization condition, is chosen.
Using the blow-up technique, it 
is possible to prove that there exist a sequence $(U_{k})$ of Talenti 
instantons, concentrating at the boundary of $\Om$, such that 
the 
$H^1$ norm of the difference between $u_{k}$ 
and $U_{k}$ approaches zero, as $k\to+\infty$.  The value of
$\Psi_{\alpha_{k}}(U_{k})$ can be used to
estimate 
$\Psi_{\alpha_{k}}(u_{k})$ from below.  However, 
$\Psi_{\alpha_{k}}(U_{k})>\sddn$ for large $k$.  This contradicts
the hypothesis that $\alpha_{0}=+\infty$.
We use this argument to prove our family of inequalities.
We remark that in the present analysis the functional 
$\Phi_{\alpha}$ in (\ref{phivelho}) is replaced by
$\Phi_{\alpha}:H^1(\Om)\setminus\{0\}\to\Rb$ defined by
$$
\Phi_{\alpha}(u)=\lb
\frac 12\nhu-\frac{1}{\ts}\ius
\rb(1+\alpha\delta(u))^{\frac{N}{2}},
$$
where $\delta:H^1(\Om)\setminus\{0\}\to\Rb$, depending
on $q$, is homogeneous
of degree zero.  This leads to the problem
$$
  \left\{\begin{array}{rcll}\!\!\!
   \lb 1+\frac{\as}{2}\alpha\delta(u)\rb(-\Delta u+au)&&&\\+
   \frac{q\bt}{2}\alpha |u|_{q}^{q(\bt-1)}|u|^{q-2}u
   &\!\!=\!\!&
   \lb 1+\lb1+\as\frac{\ts}{4}\rb\alpha\delta(u)\rb 
   |u|^\tsd u&\mbox{in\ }\Om,\\
    u&\!\!>\!\!&0
    &\mbox{in\ }\Om,\\ 
    \frac{\partial u}{\partial\nu}&\!\!=\!\!&0
   &\mbox{on\ }\partial\Om,
   \end{array}\right.
$$
where $\as\in[0,1]$ and $\bt\in\left[\frac{2}{\tf},1\right]$
are constants which depend on $q$ and $N$.

Our approach is based on the work
\cite{APY}, due to Adimurthi, Pacella and 
Yadava.
We use \cite{AM}, \cite{G}, \cite{WX} and \cite{Zh},
already mentioned.  Of course,
Talenti~\cite{T},
Brezis and Nirenberg~\cite{BN} and P.L.\ 
Lions~\cite{L1} are also of major importance.
To our knowledge,
Hebey and Vaugon~\cite{HV} were the first 
to use a contradiction argument based on blow-up
estimates to obtain sharp Sobolev inequalities.
We refer to 
Adimurthi and Yadava~\cite{AY1},
Brezis and Lieb~\cite{BL}, 
Chabrowski and Willem~\cite{CW},
Li and Zhu~\cite{LZ},
Lions, Pacella and Tricarico~\cite{LPT},
Z.Q.\ Wang~\cite{W2,W1} and
M.\ Zhu~\cite{zhu}
for related results.

The organization of this work is as follows.
In Section~2 we introduce a family of 
functionals, derive their associated Euler equations
and state our main theorem.
In Section~3, arguing by contradiction,
we assume that least energy critical points exist for all positive 
values of $\alpha$ and analyze their asymptotic behavior.
In Section~4 we prove our main theorem.
Finally, in the Appendix we prove a 
technical estimate similar to those in 
Adimurthi and Mancini~\cite{AM}.

\section{The functionals and their associated Euler
equations}

Let $N\geq 5$, $a>0$, $\alpha\geq 0$ and $\Om$ be a smooth 
bounded domain in $\Rb^{N}$.
We regard $a$ as fixed and 
$\alpha$ as a parameter.
Denote the $L^p$ and $H^1$ norms of $u$ in $\Om$ by
$$|u|_{p}:=\left(\textstyle\int|u|^p\right)^\frac{1}{p}\qquad  
\mbox{and}\qquad
||u||:=\left(\ngupq+a\nupq\right)^\frac{1}{2}.$$
Unless otherwise indicated,
integrals are over $\Om$. 

Let
$$\ts=\ts(N):=
\frac{2N}{N-2}$$
be the critical exponent for the Sobolev embedding
$H^1(\Om)\subset L^q(\Om)$,
$$\tf=\tf(N):=\frac{2N}{N-1}\qquad\mbox{and}\qquad
\tz=\tz(N):=\frac{2(N-1)}{N-2}.$$
We consider $q$ such that $$\tf\leq q\leq\tz,$$
and define $\as\in[0,1]$ and $\bt\in\left[\frac{2}{\tf},1\right]$
by
\be\label{as}\txt\as=2-N+\frac{q}{\ts-q}
\ee
and
\be\label{bt}
\txt\bt=\frac{2}{N-2}\cdot\frac{1}{\ts-q}.\ee
We easily check that\footnote{The reader can also 
verify that $s=\mbox{\tiny $(N-1)$}
\times\frac{\,\,\,q-\tf}{\,\ts-q}$ and
$t=\frac{s}{N}+\frac{N-1}{N}$.}
\be\label{qba}\txt
q\bt=\frac{2}{N-2}\cdot\as+2.\ee
Moreover,
$$\begin{array}{lllll}
q=\tf&\implies&\as=0&\mbox{and}&\bt=\frac{2}{\tf},\\
q=\tz&\implies&\as=1&\mbox{and}&\bt=1.
\end{array}$$

We recall that the 
infimum
$$
S:=\inf_{\stackrel{
\mbox{\tiny $u\in
L^\ts(\Rb^{N})\setminus\{0\}$}} 
{\mbox{\tiny $\nabla u\in L^2(\Rb^{N})$}
}}
\frac{\int_{\Rb^{N}}|\nabla 
u|^2}{\lb\int_{\Rb^{N}}|u|^\ts\rb^{2/\ts}},
$$
 which depends on $N$, is achieved by the Talenti instanton
$$
U(x):=\left(\frac{N(N-2)}{N(N-2)+|x|^2}\right)^{\frac{N-2}{2}}.
$$
This instanton $U$ satisfies 
\begin{equation}\label{inst}
-\Delta U=U^{\ts-1},
\end{equation} so that
\begin{equation}\label{insta}\int_{\Rb^{N}}
    |\nabla U|^2=\int_{\Rb^{N}} U^\ts=S^{\frac{N}{2}}=
    [N(N-2)]^\frac{N}{2}\omega_{N}
    \frac{1}{2^{N}}\sqrt{\pi}
    \frac{\Gamma\left(\frac{N}{2}\right)}{\Gamma\left(\frac{N+1}{2}\right)}.
\end{equation}
The value $\omega_{N}$ is the volume of the $N-1$ dimensional 
unit sphere:
    $$
    \omega_{N}=\frac{2\pi^\frac{N}{2}}{\Gamma\lb
    \frac{N}{2}\rb}.$$ 
Substituting this value in the previous equation,  
    $$S^{\frac N2}=\frac{\pi^\frac{N+1}{2}}{2^{N-1}}\cdot
    \frac{[N(N-2)]^\frac{N}{2}}{\Gamma\lb\frac{N+1}{2}\rb}.
    $$
Let $\eps>0$ and $y\in\Rb^{N}$.  
We define the rescaled instanton
\be
U_{\eps,y}:=\eps^{-\frac{N-2}{2}}U\left(
\frac{x-y}{\eps}\right),\label{resins}
\ee
which also satisfies (\ref{inst}) and (\ref{insta}).

We are interested in studying the $C^2$ functionals
$\Psi_{\alpha}:H^1(\Om)\setminus\{0\}\to\Rb$, defined
by
\be\label{psiu}
\Psi_{\alpha}(u):=
\frac{\nhu}{\nupsq}\lb1+\alpha\defdeli\rb.
\ee

We regard $\Psi_{\alpha}$ as a restricted functional, 
in following sense.
Consider the 
functionals $\beta$ and $\delta:H^1(\Om)\setminus\{0\}\to\Rb$, 
homogeneous of degree zero, defined by
$$
\beta(u):=\frac\nhu\nupsq
$$
and
$$
\delta(u):=\defdeli.
$$
 We can write $\Psi_{\alpha}$ in terms of $\alpha$,
$\beta$ and $\delta$ as
$$
\Psi_{\alpha}=\beta(1+\alpha\delta).
$$
 Consider also the $C^2$ functionals
$\Phi_{\alpha}:H^1(\Om)\setminus\{0\}\to\Rb$, defined by
 $$
\Phi_{\alpha}(u):=\lb
\frac 12\nhu-\frac{1}{\ts}\ius
\rb\lb1+\alpha\defdeli\rb^{\frac{N}{2}}=
\Phi_{0}(u)(1+\alpha\delta(u))^{\frac{N}{2}}.
$$
 We recall that the Nehari manifold is
$${\mathcal{N}}:=\left\{u\in H^1(\Om)\setminus\{0\}:
\Phi_{\alpha}^\prime(u)u=0\right\}
=\left\{u\in H^1(\Om)\setminus\{0\}:||u||^2=\ius\right\}.$$
For any $u\in H^1(\Om)\setminus\{0\}$, there exists a 
unique $\tau(u)>0$ such that 
$\tau(u)u\in{\mathcal{N}}$.  The value of $\tau(u)$ is
$$
\tau(u)=\lb\frac{\nhu}{\ius}\rb^\frac{N-2}{4}
$$
and
$$
{\txt\frac{1}{N}}\lb\Psi_{\alpha}(u)\rb^\frac{N}{2}=
\Phi_{\alpha}(\tau(u)u).
$$

Next we derive the Euler equation 
associated to $\Phi_{\alpha}$.  Since
$$
\Phi_{\alpha}^\prime=(1+\alpha\delta)^{\frac{N}{2}-1}\left[
\Phi_{0}^\prime(1+\alpha\delta)+{\textstyle \frac{N}{2}}
\Phi_{0}\alpha\delta'
\right]
$$
and
\bas\delta'(u)(\phi)&=&
-\ (2-\as)
\frac{\delta(u)}{\nhu}\intum\\
&&+\ q\bt\frac{\delta(u)}{|u|_{q}^q}
\int (|u|^{q-2} u\phi)\\
&&-\ \frac{\ts}{2\,}\as
\frac{\delta(u)}{\ius}\int (|u|^{\tsd}u\phi),
\eas
for all $\phi\in H^1(\Om)$, 
the critical points of $\Phi_{\alpha}$ satisfy
\addtocounter{equation}{+1}
\setcounter{dois}{\theequation}
\nc{\rfd}{$(\thedois)_{\alpha}$}
$$
\begin{array}{rcl}{\txt
\left[1+\alpha\delta(u)\lb 1-(2-\as)\frac{N}{4}+(2-\as)\frac{N-2}{4}
\frac{\ius}{\nhu}\rb\right]}{\displaystyle
\intum} &&
\\
+\ {\txt\left[
\lb\frac 12\lb\frac{||u||}{|u|_{\ts}^{\ts\!/2}}\rb^\as-
\frac{1}{\ts}\lb\frac{|u|_{\ts}^{\ts\!/2}}{||u||}\rb^{2-\as}\rb\right]
\frac{q\bt N}{2}}{\displaystyle\alpha |u|_{q}^{q(\bt-1)}
\int (|u|^{q-2} u\phi)}&&\\ 
-\ \left[1+\alpha\delta(u)\lb 1-\as\frac{N}{4}+\as\frac{\ts 
N}{8}\frac{\nhu}{\ius}\rb\right]
{\displaystyle\int (|u|^{\tsd}u\phi)}&=&0,
\end{array}\eqno{(\theequation)_{\alpha}}
$$
for all $\phi\in H^1(\Om)$.
However, 
this equation can be simplified.
By taking $\phi=u$,
i.e., by differentiating $\Phi_{\alpha}$ along the radial 
direction,
we deduce that $||u||^2=\ius$.
So the critical points of $\Phi_{\alpha}$ satisfy
\addtocounter{equation}{+1}
\setcounter{um}{\theequation}
\nc{\rf}{$(\theum)_{\alpha}$}
\nc{\rfk}{$(\theum)_{\alpha_{k}}$}
\nc{\rfz}{$(\theum)_{\alpha_{0}}$}
$$
  \left\{\begin{array}{rcll}\!\!\!
   \lb 1+\frac{\as}{2}\alpha\delta(u)\rb(-\Delta u+au)&&&\\+
   \frac{q\bt}{2}\alpha |u|_{q}^{q(\bt-1)}|u|^{q-2}u
   &\!\!=\!\!&
   \lb 1+\lb1+\as\frac{\ts}{4}\rb\alpha\delta(u)\rb 
   |u|^\tsd u&\mbox{in\ }\Om,\\
   \frac{\partial u}{\partial\nu}
   &\!\!=\!\!&
   0&\mbox{on\ }\partial\Om.
   \end{array}\right.\eqno{(\theequation)_{\alpha}}
$$
Conversely, we now check that the solutions of \rf\
are 
solutions of
\rfd, i.e.\ the solutions of \rf\ satisfy
\be\label{igual}
||u||^2=|u|_{\ts}^\ts.
\ee
By multiplying \rf\
by $u$ and integrating 
over $\Om$ we get 
$$
\lb 1+{\txt\frac{\as}{2}}\alpha\delta(u)\rb||u||^2+
{\txt\frac{q\bt}{2}}\alpha|u|_{q}^{q\bt}=
\lb
1+\lb 1+\as{\txt\frac{\ts}{4}}\rb\alpha\delta(u)
\rb|u|_{\ts}^\ts
$$
or
\be\label{eqgamma}
\lb 1+{\txt\frac{\as}{2}}\alpha\delta(u)\rb
{\txt\lb\frac{||u||}{|u|_{\ts}^{\ts\!/2}}\rb^2}+
{\txt\frac{q\bt}{2}}\alpha\delta(u)
{\txt\lb\frac{||u||}{|u|_{\ts}^{\ts\!/2}}\rb^{2-\as}}=
1+\lb 1+\as{\txt\frac{\ts}{4}}\rb\alpha\delta(u).
\ee
Let
$$
c_{1}:=1+{\txt\frac{\as}{2}}\alpha\delta(u),
\qquad
c_{2}:=
{\txt\frac{q\bt}{2}}\alpha\delta(u)
$$
and
$$\gamma:=\frac{||u||}{|u|_{\ts}^{\ts\!/2}}.$$
Equation (\ref{qba}) implies that
$$
c_{1}+c_{2}=1+\lb 1+\as{\txt\frac{\ts}{4}}\rb\alpha\delta(u)
$$
Hence, we can write (\ref{eqgamma}) as 
$$c_{1}\gamma^2+c_{2}\gamma^{2-\as}=c_{1}+c_{2}.$$
Therefore $\gamma$ has to be one,
and the solutions of \rf\ are
solutions of \rfd.

The critical points of $\Psi_{\alpha}$ satisfy
$$
  \left\{\begin{array}{rcll}\!\!\!
   \lb 1+\frac{\as}{2}\alpha\delta(u)\rb
   \frac{(-\Delta u+au)}{\nhu}&&&\\+
   \frac{q\bt}{2}\alpha\delta(u)\frac{|u|^{q-2}u}{|u|_{q}^q}
   &\!\!=\!\!&
   \lb 1+\lb1+\as\frac{\ts}{4}\rb\alpha\delta(u)\rb 
   \frac{|u|^\tsd u}{\ius}&\mbox{in\ }\Om,\\
   \frac{\partial u}{\partial\nu}
   &\!\!=\!\!&
   0&\mbox{on\ }\partial\Om.
   \end{array}\right.
$$
 If $u$ is a critical point
of $\Phi_{\alpha}$, then every nonzero multiple of $u$,
in particular $u$, is a 
critical point of $\Psi_{\alpha}$.  
Conversely, if $u$ is a critical point of
$\Psi_{\alpha}$, then 
$\tau(u)u$ is a critical point of
$\Phi_{\alpha}$.
We are interested in proving existence and nonexistence of 
{\em least 
energy\/} critical points of $\Phi_{\alpha}$,
or equivalently of $\Psi_{\alpha}$.  
We recall that a least energy critical point of 
$\Phi_{\alpha}$
is a function $u\in H^1(\Om)\setminus\{0\}$, such that
$$
\Phi_{\alpha}(u)=\inf_{{\mathcal{N}}}\Phi_{\alpha}=
\inf_{H^1(\Om)\setminus\{0\}}
{\textstyle\frac{1}{N}} (\Psi_{\alpha})^\frac{N}{2}.
$$

\begin{rem} System \rf\ possesses one and only one constant solution
$u\equiv a^\frac{N-2}{4}$.
\end{rem}

Our main result is
\begin{thm}\label{theorem}
    Let $N\geq 5$,
   $\Om$ be a smooth bounded domain in $\Rb^{N}$, 
$a>0$, $\alpha\geq 0$ and $\tf\leq q\leq\tz$. 
There exists a positive real number
$\alpha_{0}=\alpha_{0}(q,a,\Om)$ such that
\begin{enumerate}
    \item[{\rm (i)}] if $\alpha<\alpha_{0}$, then
    $\Psi_{\alpha}$ has a least energy critical point $u_{\alpha}$;
    $\Psi_{\alpha}(u_{\alpha})<\sddn$;
    \item[{\rm (ii)}] if 
    $\alpha>\alpha_{0}$, then $\Psi_{\alpha}$ does not have a least
energy 
    critical point and
    $$
    \sddn=\inf_{H^1(\Om)\setminus\{0\}}\Psi_{\alpha}.
    $$
\end{enumerate}
\end{thm}
This theorem obviously implies that $\sddn\leq\Psi_{\alpha_{0}}$, 
i.e.,
$$\sddn\nupsq\leq\nhu+\alpha_{0}\qa\uqb,$$
for all 
$u\in H^1(\Om)\setminus\{0\}$.

\begin{rem}\label{lbum}
    It is easy to check that
    $$\Psi_{\alpha}(1)=a|\Om|^\frac{2}{N}
    \lb 1+\frac{\alpha}{a^\frac{2-\as}{2}|\Om|^{1-\bt}}\rb.
    $$
    So, if
    $$
    a\leq\frac{S}{(2|\Om|)^\frac{2}{N}},
    $$
    the least energy critical points of $\Psi_{\alpha}$ might
    be constant for $\alpha$ such that
    $\Psi_{\alpha}(1)\leq\sddn$, i.e.
    $$\alpha\leq|\Om|^{1-\bt} \cdot
    \frac{{S}/{\lb 2|\Om|\rb^\frac{2}{N}}-a}{a^{\as/2}}.
    $$
    This simple observation yields the following lower bound for
    $\alpha_{0}$: $$\alpha_{0}\geq
    |\Om|^{1-\bt} \cdot
    \frac{{S}/{\lb 2|\Om|\rb^\frac{2}{N}}-a}{a^{\as/2}}.
    $$
    A second lower bound for $\alpha_{0}$ is given in\/
   {\rm Lemma~\ref{lbdois}}.
\end{rem}
\begin{rem}
    Let $\kappa>0$.  By scaling, we easily check that
    $$\alpha_{0}\lb q, a\,\kappa^2,\frac{\Om}{\kappa}\rb=\kappa\,
    \alpha_{0}(q,a,\Om).$$
    In fact, if $u\in H^1(\Om)$ and $v:\frac{\Om}{\kappa}
    \to\Rb$ is defined by
    $v(x)=\kappa^\frac{N-2}{2}u(\kappa x)$, then
    $v\in H^1\lb\frac{\Om}{\kappa}\rb$ satisfies
$$\begin{array}{lcl}   
    \kappa^2|v|_{L^2\lb\frac{\Om}{\kappa}\rb}^{2}&=&|u|_{2}^{2},\\
    \kappa|v|_{L^q\lb\frac{\Om}{\kappa}\rb}^{q\bt}&=&
    |u|_{q}^{q\bt},\\
    |v|_{L^\ts\lb\frac{\Om}{\kappa}\rb}&=&|u|_{\ts},\\
    |\nabla v|_{L^2\lb\frac{\Om}{\kappa}\rb}&=&|\nabla u|_{2}.
\end{array}$$
\end{rem}

\section{Asymptotic behavior of least energy 
critical points}

We consider the minimization problem 
corresponding to
$$S_{\alpha}:=\inf_{
H^1(\Om)\setminus\{0\}}\Psi_{\alpha}.$$
From Adimurthi and Mancini \cite{AM} and X.J.\ Wang~\cite{WX},
we know that 
\be\label{szero}
0<S_{0}<\sddn.
\ee
Obviously, $S_{\alpha}$ is 
nondecreasing as $\alpha$ increases.  Choose any point
$P\in\partial\Om$.
By testing $\Psi_{\alpha}$ with $U_{\eps,P}$ and letting $\eps\to 0$,
we 
conclude that $S_{\alpha}\leq\sddn$ for all $\alpha\geq 0$.

\begin{rem}\label{pscondition}
    If $S_{\alpha}<\sddn$, then $S_{\alpha}$ is achieved.
\end{rem}
    We can assume the minimizer is a
    nonnegative function.  In fact, by the maximum principle, a 
    nonnegative minimizer is positive in $\Om$.
\begin{rem} 
    The map $\alpha\mapsto S_{\alpha}$ is continuous on $[0,+\infty[$.
\end{rem}
The proof of this remark is similar to the one of Lemma~3.2 of
\cite{G}.

By the previous remark, the value
\be\label{min}
\alpha_{0}:=\sup\left\{\alpha\in\Rb:S_{\alpha}<\sddn\right\}
\ee
is well defined.  By (\ref{szero}) it is not zero.
Remark~\ref{pscondition} implies
\begin{rem}\label{coro}
    The map $\alpha\mapsto S_{\alpha}$ is strictly increasing on 
    $[0,\alpha_{0}]$.
    If $\alpha\in]\alpha_{0},+\infty[$, then $\Psi_{\alpha}$ does 
    not have a least energy critical point.
\end{rem}

Therefore, to prove Theorem~\ref{theorem}
we just have to establish that $\alpha_{0}$ is finite.  
Arguing by contradiction,
we assume that the value $\alpha_{0}$ in 
Theorem~\ref{theorem} is infinite and analyze the asymptotic 
behavior of least energy critical points as $\alpha\to+\infty$.

\begin{lem}\label{infinito}  The limit of $S_{\alpha}$ as
    $\alpha$ tends to $+\infty$ is
    \be\label{salpha} \lim_{\alpha\to+\infty}S_{\alpha}=\sddn.\ee  
    Suppose $S_{\alpha}<\sddn$ for all $\alpha\geq 0$.  Choose a
    sequence $\alpha_{k}\to+\infty$ as $k\to+\infty$ and let
    $\uk$ be a minimizer for $\Psi_{\alpha_{k}}$ satisfying 
    \rfk.  
    The sequence $(\uk)$ satisfies
     $$\uk\weak 0\mbox{\ in\ }H^1(\Om),$$
       \be\label{sndd}
       \lim_{k\to\infty}\ngukpq=\lim_{k\to\infty}\nukpss=\sndd
       \ee
       and
       \be\label{limakdk}
    \lim_{k\to\infty}\ak\delta(\uk)=
    0.
    \ee     
    If we denote by 
    \ba\mk&:=&\max_{\bar\Om}\uk\label{mekp}\\
    \noalign{\noindent\mbox{and}}
    \beps_{k}&:=& M_{k}^{-\frac{2}{N-2}},\label{beps}\\
    \noalign{\noindent\mbox{then}}
    \label{mtoinfty}\mk&\to&+\infty\\
    \noalign{\noindent\mbox{and}}
    \label{akdk}\ak\beps_{k}&\to&0,\ea
    as $k\to\infty$.
\end{lem}
\begin{proof}
    Suppose $S_{\alpha}<\sddn$ for all $\alpha\geq 0$ and choose a
    sequence $\alpha_{k}\to+\infty$ as $k\to+\infty$.  Let $\uk$ 
    be a minimizer for $\Psi_{\alpha_{k}}$ satisfying 
    \rfk, which necessarily exists by 
    Remark~\ref{pscondition} and rescaling.  The functions $\uk$
    satisfy
    $$
    \frac{||\uk||^2}{|u_{k}|_{\ts}^2}=\beta(\uk)<
    \Psi_{\alpha_{k}}(u_{k})<\sddn
    $$
    and 
    \be\label{igualk}
    ||\uk||^2=|\uk|_{\ts}^\ts,
    \ee
    because of (\ref{igual}).  Together,
    $$
    ||\uk||^\frac{4}{N}<\sddn,
    $$
    the sequence $\uk$ is bounded in $H^1(\Om)$.  
    
    The definition of $\Psi_{\alpha}$
    (equality (\ref{psiu})) and
    (\ref{igualk})
    imply that
    $$\ak\frac{||u_{k}||^\as|u_{k}|_{q}^{q\bt}}{
    |u_{k}|_{\ts}^{2+\ts\as/2}}
    =\ak\frac{|u_{k}|_{q}^{q\bt}}{
    |u_{k}|_{\ts}^{2}}
    <\sddn,$$
    for all positive integers $k$.  If we combine this inequality
    with the fact that the norms $|u_{k}|_{\ts}$ 
    are uniformly bounded we deduce that 
    $\uk\weak 0$ in $H^1(\Om)$.
    We can assume that $\uk\to 0$ a.e.\ on 
    $\Om$, and $|\nabla\uk|^2\weak\mu$ and
    $|\uk|^\ts\weak\nu$ in the sense of measures on $\bar\Om$.  
    So,
    $$
    \lim_{k\to\infty}\ngukpq=||\mu||
    $$
    and
    $$
    \lim_{k\to\infty}\nukpss=||\nu||,
    $$
    where 
    $$
    \sddn||\nu||^\frac 2\ts\leq||\mu||.
    $$
    Now equality (\ref{salpha}) follows from
    \be\label{auxu}
    \sddn
    \leq\frac{||\mu||}{||\nu||^\frac 2\ts}
    =\lim_{k\to\infty}\beta(u_{k})
    \leq\lim_{k\to\infty}\Psi_{\alpha_{k}}(u_{k})
    =\lim_{k\to\infty}S_{\ak}
    \leq\sddn.
    \ee
    Taking the limit of both sides of (\ref{igualk}) as
    $k\to+\infty$,
     \be\label{auxd}||\nu||=||\mu||.\ee
    Combining (\ref{auxu}) and (\ref{auxd}),
    $$
    ||\mu||=||\nu||=\sndd,
    $$
    or (\ref{sndd}).
    
    Equalities (\ref{sndd}) imply there exists a 
    constant $c$ such that 
    \be\label{c}
    |\uk|_{\ts}\geq c>0,
    \ee
    for all positive integers $k$.
    Another consequence of (\ref{auxu}) is that
    $\lim_{k\to\infty}\beta(u_{k})=\sddn$
    and so
    $$\lim_{k\to\infty}\ak\beta(\uk)\delta(\uk)=0.$$
    However,
    $$\lim_{k\to\infty}\ak\beta(\uk)\delta(\uk)=
    \sddn\lim_{k\to\infty}\ak\delta(\uk).$$ 
    Equality (\ref{limakdk}) follows.
   
    Combining (\ref{limakdk}), 
    $$\ak\delta(\uk)=
    \ak\frac{|u_{k}|_{q}^{q\bt}}{
    |u_{k}|_{\ts}^{2}}$$
    and the fact that the norms $|\uk|_{\ts}$
    are uniformly bounded, we also get
    \be\label{aq}
    \lim_{k\to\infty}\ak|u_{k}|_{q}^{q\bt}=0.
    \ee
    But, from (\ref{bt}),
    \bas
    |u_{k}|_{q}^{q\bt}&=&\lb\int u_{k}^q\rb^\bt\\
    &=&M_{k}^{q\bt}\left[\int\lb\frac{u_{k}}
    {M_{k}}\rb^q\right]^\bt\\
    &\geq&M_{k}^{q\bt}\left[\int\lb\frac{u_{k}}
    {M_{k}}\rb^\ts\right]^\bt\\
    &=&M_{k}^{(q-\ts)\bt}\lb\int u_{k}^\ts\rb^\bt\\
    &=&M_{k}^{-\frac{2}{N-2}}|u_{k}|_{\ts}^{\ts\bt}\\
    &=&\beps_{k}|u_{k}|_{\ts}^{\ts\bt}.
    \eas
    This, (\ref{c}) and (\ref{aq}) imply
    (\ref{mtoinfty}) and (\ref{akdk}).
    \end{proof}    
\begin{rem}\label{akfinito}
    Suppose that $\ak$ converges to a positive real number and
    $S_{\alpha_{k}}\nearrow\sddn$.
    Let $\uk\in H^1(\Om)$ be a 
    minimizer for $\Psi_{\alpha_k}$ satisfying 
    \rfk and suppose $u_{k}\weak 0$ in $H^1(\Om)$.  The previous 
    argument shows that\/ {\rm (\ref{sndd})}, 
    {\rm (\ref{limakdk})}, {\rm (\ref{mtoinfty})} and\/
    {\rm (\ref{akdk})} hold.
\end{rem}
\begin{lem}\label{instantao}  
       Suppose $S_{\ak}<\sddn$ and either $\ak\to+\infty$, 
       or the hypothesis of\/ {\rm Remark~\ref{akfinito}} hold.
       Let $\uk\in H^1(\Om)$ be a positive minimizer for
$\Psi_{\alpha_k}$ 
       satisfying \rfk.  
       Then
       \be\label{convum}
       \lim_{k\to\infty}|\nabla\uk-\nabla U_{\beps_{k},P_{k}}|_{2}=0
       \ee
       and $\pkk\in\partial\Om$, for large $k$, where
       $P_{k}$ is such that $\uk(\pkk)=M_{k}$, and 
       $M_{k}$ and $\beps_{k}$ are as
       in\/ {\rm (\ref{mekp})} and\/ {\rm (\ref{beps})}, respectively.
       \end{lem}
\begin{proof}  
     We use the Gidas and Spruck blow up technique~\cite{GS}.
     Let $\Om_{k}:=(\Om-\pkk)/\beps_{k}$ and $v_{k}:\Om_{k}\to\Rb$
     be defined by
     $v_{k}(x):=\beps_{k}^\ndd \uk(\beps_{k} x+\pkk)$.
     We can assume that $P_{k}\to P_{0}$ and
     $
     \ \Om_{k}\to\Om_{\infty}$.
     We let $L=\lim_{k\to+\infty}
     \mbox{dist\,}(P_{k},\partial\Om)/
     \beps_{k}\in[0,+\infty]$. 

     From
     \bas
         |v_{k}|_{L^q(\Om_{k})}^{q\bt}&=&       
         \beps_{k}^{\frac{N-2}{2}q\bt}\beps_{k}^{-N\bt}|u_{k}|_{q}^{q\bt}\\
         &=&\beps_{k}^{-1}|u_{k}|_{q}^{q\bt},
     \eas
     we deduce that 
     \be\label{dvu}
     \delta(u_{k})=\beps_{k}\delta(v_{k}),
     \ee where the norms in
     $\delta(v_{k})$ are computed in $\Om_{k}$.  Also,
     \bas
        |v_{k}|_{L^q(\Om_{k})}^{q(\bt-1)}v_{k}^{q-1}(x)&=&
        \beps_{k}^{\frac{N-2}{2}q(\bt-1)}\beps_{k}^{-N(\bt-1)}
        \beps_{k}^{\frac{N-2}{2}(q-1)}|u_{k}|_{q}^{q(\bt-1)}u_{k}^{q-1}
        (\beps_{k}x+P_{k})\\
        &=&\beps_{k}^\frac{N}{2}|u_{k}|_{q}^{q(\bt-1)}u_{k}^{q-1}
        (\beps_{k}x+P_{k}).
     \eas
     Thus,
     \bas
     \Delta v_{k}(x)&=&\beps_{k}^\frac{N+2}{2}
     \Delta u_{k}(\beps_{k} x+\pkk)\\
     v_{k}^{\ts-1}(x)&=&\beps_{k}^\frac{N+2}{2}
     u_{k}^{\ts-1}(\beps_{k} x+\pkk)\\
     \beps_{k}^2 v_{k}(x)&=&\beps_{k}^\frac{N+2}{2}
     u_{k}(\beps_{k} x+\pkk)\\
     \beps_{k}|v_{k}|_{L^q(\Om_{k})}^{q(\bt-1)}
     v_{k}^{q-1}(x)&=&
     \beps_{k}^\frac{N+2}{2}|u_{k}|_{q}^{q(\bt-1)}
     u_{k}^{q-1}(\beps_{k} x+\pkk)
     \eas
     The functions $v_{k}$ satisfy
\be\label{return}
  \left\{\begin{array}{rcll}
   \lb 1+\frac{\as}{2}\beps_{k}\alpha_{k}\delta(v_{k})\rb
   (-\Delta v_{k}+a\beps_{k}^2v_{k})&&&\\+
   \frac{q\bt}{2}\beps_{k}\alpha_{k} 
   |v_{k}|_{L^q(\Om_{k})}^{q(\bt-1)}v_{k}^{q-1}&&&\\ -
   \lb 1+\lb1+\as\frac{\ts}{4}\rb\beps_{k}\alpha_{k}\delta(v_{k})\rb 
   v_{k}^\tsu&=&0&\mbox{in\ }\Om_{k},\\
    0<v_{k}\leq v_{k}(0)&=&
    1&\mbox{in\ }\Om_{k},\\
   \frac{\partial v_{k}}{\partial\nu}
   &=&
   0&\mbox{on\ }\partial\Om_{k}.
   \end{array}\right.
   \ee
   
   Suppose that $L=+\infty$.  Then $\Om_{\infty}=\Rb^{N}$.
   We use (\ref{limakdk}), (\ref{akdk})
   (which obviously implies $\beps_{k}\to 0$),
   (\ref{dvu}) and
   \be\label{lbq}
   |v_{k}|_{L^q(\Om_{k})}^q=\int_{\Om_{k}}
   v_{k}^q\geq\int_{\Om_{k}}v_{k}^\ts=\int u_{k}^\ts=
   |u_{k}|_{\ts}^\ts\geq c^\ts,
   \ee
   (from (\ref{c})).
   By the elliptic estimates in~\cite{ADN}, 
   $$\vk\to v\mbox{\ in\ } 
   C^2_{\mbox{\tiny loc}}(\Om_{\infty})$$
   where $v$ satisfies
   $$
   \left\{\begin{array}{ll}
   -\Delta v=v^\tsu&\mbox{in\ 
   }\Om_{\infty},
   \\
   0<v\leq v(0)=1&\mbox{in\ }\Om_{\infty}.
   \end{array}\right.
   $$
            By lower 
            semicontinuity of the norm, $v\in L^\ts(\Om_{\infty})$ and
            $\nabla v\in L^2(\Om_\infty)$.  
            Therefore
            $v=U$.  
            From (\ref{sndd}),
            $$
            S^\frac N2=\int_{\Rb^{N}}|\nabla 
            U|^2\leq\lim_{k\to\infty}|\nabla\uk|_{2}^2=\sndd,$$
            which is impossible.  

     So $L$ is finite.  This implies that $P_{0}\in\partial\Om$.
     Without loss of generality, we assume that $P_{0}=0$ and
     that in a neighborhood 
     $B_{R}(0)=\{x\in\Rb^{N}: |x|<R\}$ of 0 the sets $\Om$
     and $\partial\Om$ are described by
     \bas
     \Om\cap B_{R}(0)&=&\{(x',x_{N})\in B_{R}(0)|\,x_{N}>g(x')\},\\
     \partial\Om\cap B_{R}(0)&=&\{(x',x_{N})\in B_{R}(0)|
     \,x_{N}=g(x')\},    
     \eas
     where $g:B_{R}(0)\cap\{(0,x_{N})|\, x_{N}\in\Rb\}\to\Rb$ is such
     that $g(0)=0$ and $\nabla g(0)=0$.  We make the change of
     coordinates associated to
     the map $\psi=(\psi_{1},\ldots,\psi_{N}):B_{R}(0)\to\Rb^{N}$,
with
     \bas
     \psi_{i}(x)&=&x_{i}
     -\frac{g(x')-x_{N}}{1+|\nabla g(x')|^2}\cdot
     \frac{\partial g}{\partial x_{i}}(x'),\ \ \mbox{for\ }1\leq i\leq N-1,\\
     \psi_{N}(x)&=&x_{N}-g(x').
     \eas
     The determinant of the Jacobian of $\psi$ at 0 is 1.
     We can choose $R_{0}>0$ and an open neighborhood 
     $V\subset B_{R}(0)$ of zero, such that 
     $$\begin{array}{l}
     \psi:V\to B_{R_{0}}(0)\mbox{\ is a diffeomorphism},\\
     \psi:\Omega\cap V\to
     B_{R_{0}}(0)_{+}:=\{(y',y_{N})\in B_{R_{0}}|\, y_{N}>0\},\\
     \psi:\partial\Omega\cap V\to\{(y',y_{N})\in B_{R_{0}}|\, y_{N}=0\}.
     \end{array}$$
     If $u:V\to\Rb$ is smooth and $v:B_{R_{0}}(0)_{+}\to\Rb$ is such that
     $v(y)=u(\psi^{-1}(y))$, then
     $$
     (\Delta u)(\psi^{-1}(y))=\sum_{i,j=1}^{N}
     a_{i,j}(y)\frac{\partial^2v}
     {\partial y_{i}\partial y_{j}}(y)
     +\sum_{i=1}^{N}b_{i}(y)\frac{\partial v}
     {\partial y_{i}}(y),
     $$
     $$
     \frac{\partial u}{\partial\nu}(\psi^{-1}(y))=d(y)
     \frac{\partial v}{\partial y_{N}}(y)\mbox{\ \ on\ }y_{N}=0,
     $$
     with $a_{i,j}$, $b_{i}$ and $d$ smooth functions,
     $$
     a_{i,j}(y)=\delta_{i,j}+O(|y|)
     $$
     and
     $$
     d(y)=1+|(\nabla g)[(\psi^{-1}(y))']|^2\geq 1.
     $$
     As above, $(\psi^{-1}(y))'$ denotes the first $N-1$ 
     coordinates of $\psi^{-1}(y)$.
     We let $Q_{k}=\psi(P_{k})$ and denote by
     $(Q_{k})_{N}$ the $N$-th coordinate of $Q_{k}$.
     We also let $B_{k}=(B_{R_{0}}(0)_{+}-Q_{k})/\beps_{k}$.
     We define 
     $w_{k}:B_{R_{0}}(0)_{+}\to\Rb$ by
     $w_{k}(y)=u_{k}(\psi^{-1}(y))$
     and
     $\tilde{w}_{k}:B_{k}\to\Rb$ by
     $\tilde{w}_{k}(x)=\beps_{k}^\frac{N-2}{2}w_{k}(\beps_{k}x+Q_{k})$.
     The functions $\tilde{w}_{k}$ satisfy
$$
  \left\{\begin{array}{rcl}\!\!\!
   \lb 1+\frac{\as}{2}\beps_{k}\alpha_{k}
   \delta(v_{k})\rb\times
   \qquad\qquad\qquad\qquad\qquad\qquad\ \ \ &&\\
   \lb-
   {\txt\sum_{i,j=1}^{N}}\tilde{a}_{i,j,k}
   \frac{\partial^2\tilde{w}_{k}}
     {\partial x_{i}\partial x_{j}}
     -{\txt\sum_{i=1}^{N}}\beps_{k}\tilde{b}_{i,k}
     \frac{\partial\tilde{w}_{k}}
     {\partial x_{i}}
   +a\beps_{k}^2\tilde{w}_{k}
   \rb&&\\+
   \frac{q\bt}{2}\beps_{k}\alpha_{k} 
   |v_{k}|_{L^q(\Om_{k})}^{q(\bt-1)}
   \tilde{w}_{k}^{q-1}&&\\   
   -\lb 1+\lb1+\as\frac{\ts}{4}\rb\beps_{k}\alpha_{k}
   \delta(v_{k})\rb 
   \tilde{w}_{k}^\tsu
   &\!\!\!=\!\!\!&0
   \mbox{\ \ in\ }B_{k},\\
    0<\tilde{w}_{k}\leq\tilde{w}_{k}(0)&\!\!\!=\!\!\!&
    1\mbox{\ \ in\ }B_{k},\\
   \frac{\partial\tilde{w}_{k}}{\partial x_{N}}
   &\!\!\!=\!\!\!&
   0\mbox{\ \ on\ }\mbox{\b{$\partial$}} B_{k},
   \end{array}\right.
$$
   with
   $\mbox{\b{$\partial$}} B_{k}=
   \partial B_{k}\cap\lb
   \Rb^{N-1}\times\{ 
   -(Q_{k})_{N}/\beps_{k}\}\rb
   $,
   \be\label{aij}
   \tilde{a}_{i,j,k}(x)=a_{i,j}(\beps_{k}x+Q_{k})=
   \delta_{i,j}+O(|\beps_{k}x+Q_{k}|)
   \ee
   and $\tilde{b}_{i,k}(x)=b_{i}(\beps_{k}x+Q_{k})$.
   
   We use again (\ref{limakdk}), (\ref{akdk}),
   (\ref{dvu}),
   (\ref{lbq}), and we also use 
   (\ref{aij}).
      By elliptic regularity theory,
      $\tilde{w}_{k}\to w$ in 
     $C^2_{\mbox{\tiny loc}}(\bar{B}_{\infty})$
     where $B_{\infty}=\{(x',x_{N})\in\Rb^{N}:x_{N}>-L\}$ and 
   $$
   \left\{\begin{array}{ll}
   -\Delta w=w^\tsu&\mbox{in\ 
   }B_{\infty},
   \\
   0<w\leq w(0)=1&\mbox{in\ }B_{\infty},
   \\ \frac{\partial w}{\partial x_{N}}=0&
   \mbox{on\ }\partial B_{\infty}.
   \end{array}\right.
   $$
   We deduce that $w=U$.
     Moreover, $L$ has to be zero. 
     
     Suppose 
     $P_{k}\not\in\partial\Om$ for large $k$. 
     Since $\nabla\tilde{w}_{k}(0)=0$ and
     $\frac{\partial\tilde{w}_{k}}{\partial x_{N}}=0$ on 
     $\partial B_{k}\cap\lb\Rb^{N-1}\times\{ 
                -(Q_{k})_{N}/\beps_{k}\}\rb$,
     by the mean value theorem there exists
     $r_{k}\in\Rb$, with $-(Q_{k})_{N}/\beps_{k}<r_{k}<0$ such that
     $\frac{\partial^2\tilde{w}_{k}}{\partial x_{N}^2}(0,r_{k})=0$.
     Recalling that
     $\tilde{w}_{k}\to w$ in $C^2_{\mbox{\tiny loc}}(\bar B_{\infty})$, 
     it follows that
     $\frac{\partial^2w}{\partial x_{N}^2}(0)=0$.  This is impossible
     because $w=U$ and
     $\frac{\partial^2U}{\partial x_{N}^2}(0)<0$.
     We conclude that $P_{k}\in\partial\Om$ for large $k$.
     
     Returning to (\ref{return}), 
      \be\label{ctloc}\vk\to v\mbox{\ in\ } 
   C^2_{\mbox{\tiny loc}}(\Om_{\infty})\ee
     where 
     $\Om_{\infty}=\ttt$ and 
        $$
   \left\{\begin{array}{ll}
   -\Delta v=v^\tsu&\mbox{in\ 
   }\Om_{\infty},
   \\
   0<v\leq v(0)=1&\mbox{in\ }\Om_{\infty},
    \\ \frac{\partial v}{\partial\nu}=0&
    \mbox{on\ }\partial\Om_{\infty}.
   \end{array}\right.
   $$
   So $v=U$.
     Finally, from (\ref{sndd}), 
     (\ref{ctloc}) and
     $$
     \int_\ttt |\nabla U|^2=\sndd,$$ we deduce (\ref{convum}).
\end{proof}

As in \cite{APY} and \cite{BE},
let $${\mathcal{M}}:=\{CU_{\eps,y}, C\in\Rb, \eps>0, y\in\partial\Om\}$$
and $d(u,{\mathcal{M}}):=\inf\{|\nabla (u-V)|_{2}, V\in{\mathcal{M}}\}$.
The set ${\mathcal{M}}\setminus\{0\}$ is a manifold of dimension $N+1$.
The tangent space 
$T_{C_{l},\eps_{l},y_{l}}({\mathcal{M}})$ at $C_{l}U_{\eps_{l},y_{l}}$ is 
given by
$$
T_{C_{l},\eps_{l},y_{l}}({\mathcal{M}})=\mbox{span\;}\left\{
U_{{\eps},y},C\frac\partial{\partial\eps}U_{{\eps},y},
C\frac\partial{\partial\tau_{i}}U_{{\eps},y}, 1\leq i\leq N-1
\right\}_{(C_{l},\eps_{l},y_{l})}
$$
where $T_{x}(\partial\Om)=\mbox{span}\{\tau_{1},\ldots,\tau_{N-1}\}$.

As in Lemma~\ref{instantao}, let $\uk\in H^1(\Om)$ be a 
positive minimizer for $\Psi_{\alpha_k}$ satisfying \rfk.  
For large $k$,
the infimum $d(\uk,{\mathcal{M}})$ is achieved: 
\be\label{m}
d(\uk,{\mathcal{M}})=|\nabla(\uk-C_{k}U_{\eps_{k},y_{k}})|_{2}
\mbox{\ for\ }
\ck U_{\eps_{k},y_{k}}\in{\mathcal{M}}.
\ee
Furthermore,
\be\label{ck}
\ck=1+o(1)
\ee
$y_{k}\to P_{0}$ and $\eps_{k}/\beps_{k}\to 1$ (see Lemma 1 of 
\cite{BE} and Lemma 2.3 of \cite{APY}).  From (\ref{akdk}),
\be\label{akek}
\ak\ek\to 0.
\ee

We define $$\wk:=\uk-\ck U_{\eps_{k},y_{k}},$$ 
so that
\be\label{dot}
\int\nabla
U_{\eps_{k},y_{k}}\cdot\nabla w_{k}=0.
\ee
On the one hand, from (\ref{convum}), 
$$
\lim_{k\to\infty}|\nabla(\uk-\ck U_{\eps_{k},y_{k}})|_{2}=0.
$$
On the other hand, from 
Poincar\'{e}'s inequality, and the fact that both the average of
$\uk$ and the average of $\ck U_{\eps_{k},y_{k}}$, 
in $\Om$, converge to zero,
$$
\lim_{k\to\infty}|\uk-\ck U_{\eps_{k},y_{k}}|_{\ts}=0.
$$
Together,
\be\label{zero}
\lim_{k\to\infty}||w_{k}||=0.
\ee

Our next aim is the lower bound for
$|\nabla\wk|_{2}^2+c\ak
\int U_{\eps_{k},y_{k}}^{qt\!-2}\, w_{k}^2$
in Lem\-ma~\ref{dif}
where $c$ is a constant.
To obtain that lower bound we consider two eigenvalue problems.
The first one can be regarded as the limit of the second, in a sense 
made precise below.
\begin{lem}\label{ber}
    (Bianchi and Egnell~\cite{BE}, Rey~\cite{R}) The 
eigenvalue problem 
\be\left\{
\begin{array}{ll}
    -\Delta \phi=\mu U^\tsd\phi&\mbox{in\ }\ttt,\\
    \frac{\partial\phi}{\partial\nu}=0&\mbox{on\ }\partial\ttt,\\
\int_{\ttt}U^\tsd\phi^2<\infty&
\end{array}\right.\label{bie}
\ee
admits a discrete spectrum $\mu_{1}<\mu_{2}
\leq\mu_{3}\leq\ldots$ such 
that $\mu_{1}=1$, $\mu_{2}=\mu_{3}=\ldots=\mu_{N}=\tsu$ and 
$\mu_{N+1}>\tsu$. The eigenspaces $V_{1}$ and $V_{(\tsu)}$, 
corresponding to 1 and $(\tsu)$, are given by
\bas
 V_{1}&=&\mbox{span\ }U,\\
 V_{(\tsu)}&=&\textstyle\mbox{span}\left\{\left.
 \frac{\partial U_{1,y}}{\partial y_{i}}
 \right|_{y=0},\mbox{\ for\ }1\leq i\leq N-1
 \right\}.
\eas
\end{lem}

Now we let $\eps>0$, $\nu_{\eps}>0$, and $y_{\eps}\in\partial\Om$ with 
$\lim_{\eps\to 0}y_{\eps}=y_{0}$.  Let $\{\phi_{i,\eps}\}_{i=1}^\infty$
be a complete set of orthogonal eigenfunctions with eigenvalues
$\mu_{1,\eps}<\mu_{2,\eps}\leq\mu_{3,\eps}\leq\ldots$ for the weighted
eigenvalue problem
$$\left\{
\begin{array}{ll}
    -\Delta \phi+\nu_{\eps}U^{q\bt-2}_{\eps,y_{\eps}} \phi=\mu 
    U_{\eps,y_{\eps}}^\tsd\phi&\mbox{in\ }\Om,\\
    \frac{\partial\phi}{\partial\nu}=0&\mbox{on\ }\partial\Om,\\
\end{array}\right.
$$
with $\phi_{1,\eps}>0$ and
$$
\int_{\Om}U^\tsd\phi_{i,\eps}\phi_{j,\eps}=\delta_{i,j}.
$$
Let
$$
\Om_{\eps}:=(\Om-y_{\eps})/\eps.
$$
The sets $\Om_{\eps}$ converge to a half space as $\eps\to 0$.
For a function $v$ on $\Om$, we 
define $\tilde v$ on $\Om_{\eps}$ by
$$
\tilde v(x):=\eps^{\frac{N-2}2}v(\eps x+y_{\eps}).
$$

The relation between these eigenvalue problems and the one considered 
in Lemma~\ref{ber} is given in

\begin{lem}\label{ttt}
    Suppose $y_{\eps}\in\partial\Om$, 
    $\lim_{\eps\to 0}y_{\eps}=y_{0}$,
    $\lim_{\eps\to 0}\eps^{2-\as}\nu_{\eps}=0$ and
    the sets\/ $\Om_{\eps}$ converge to $\ttt$.
    Then, up to a subsequence, 
    \be\label{muimu}
    \lim_{\eps\to 0}\mu_{i,\eps}=\mu_{i}
    \ee
    and
    \be
    \lim_{\eps\to 0}\int_{\Om_{\eps}}U^\tsd(\tilde{\phi}_{i,\eps}
    -\tilde{\phi}_{i})^2=0,\label{tstdiff}
    \ee 
    for all positive integers $i$.
    The functions
    $\mu_{i}$ and $\tilde{\phi}_{i}$ satisfy
    $$\left\{
    \begin{array}{ll}
    -\Delta \tilde{\phi}_{i}=\mu_{i} U^\tsd\tilde{\phi}_{i}&\mbox{in\ 
    }\ttt,\\
    \frac{\partial\tilde{\phi}_{i}}{\partial\nu}=0&\mbox{on\ 
    }\partial\ttt,\\
    \int_{\ttt}U^\tsd\tilde{\phi}_{i}^2=1,&
    \end{array}\right.
    $$
    and the functions $\tilde{\phi}_{i}$ are supposed 
    extended to $\Rb^{N}$ by 
    reflection.
     In particular, from the previous lemma, $\mu_{1}=1$, 
     $\tilde{\phi}_{1}=CU$ for some constant $C>0$, 
     $\mu_{i}=\tsu$ for $2\leq i\leq 
     N$ and $\mu_{N+1}>\tsu$.  Also, 
     $\{\tilde{\phi}_{i}\}_{i=2}^{N}$ is in the span of
     $\left\{{\partial U_{1,y}}/{\partial y_{i}}
     \right|_{y=0},\mbox{\ for\ }1\leq i\leq N-1\}.$
\end{lem}
We postpone the proof, since it requires the following
lemma and remark.
\begin{lem}\label{tn} Suppose $y_{\eps}\in\bar\Om$,
     $\phi_{\eps}\in H^1(\Om)$,
     $$
     \int U^{q\bt-2}_{\eps,y_{\eps}}\phi_{\eps}^2\to 0
     \qquad\mbox{and}\qquad
     \int|\nabla\phi_{\eps}|^2\to 0,
     $$
    as $\eps\to 0$.
Then
$$
    \int U^\tsd_{\eps,y_{\eps}}\phi_{\eps}^2\to 0,
$$   
    as $\eps\to 0$.
\end{lem}
\begin{proof}
    We denote the average of
    $\pk$ in $\Om$ by $\pb$.  By Poincar\'{e}'s inequality,
    $$|\pk-\pb|_{\ts}\to 0.
    $$
    The limits in this proof are taken as $\eps$ approaches zero.
    So we can write $\pk=\pb+\et$, with $\et\to 0$ in $L^\ts$.
    We know that 
    $$\int\uqz(\pbq+2\pb\et+\etq)=o(1)$$ and
    we estimate the three terms on the left hand side.
    There exists a $b>0$ such that
    $$
    \int\uqz\pbq\geq b\pbq\eps^\as.
    $$
    Also,
    \bas
    \left|\int\uqz\et\pb\right|&\leq&|\et|_{\ts}|\pb|\left(\int 
    U_{\eps,y_{\eps}}^{(q
    \bt-2)\frac{2N}{N+2}}\right)^\frac {N+2}{2N}\\
    &\leq& C|\et|_{\ts}|\pb|\eps^\as.
    \eas
    If $\tf\leq q<\tz$, then 
    \ba
    \int\uqz\etq&\leq&|\et|_{\ts}^2\left(\int U_{\eps,y_{\eps}}^{
    (q\bt-2)\frac 
    {N}{2}}\right)^\frac 2N\nonumber\\
    &\leq& C|\eta_{\eps}|_{\ts}^2\eps^\as.\label{iuzdi}
    \ea
    If $q=\tz$, then
    \ba
    \int\uqz\etq&\leq&|\et|_{\ts}^2\left(\int U_{\eps,y_{\eps}}^{\frac 
    {N}{N-2}}\right)^\frac 2N\nonumber\\
    &\leq& C|\et|_{\ts}^2\eps|\log\eps|^\frac 2N.\label{iuzd}
    \ea
    Thus,
    $$
    b\pbq\eps^\as\leq C|\pb|\eps^\as+o(1).$$
       This shows that $\pb\eps^\frac{\as}{2}$ is bounded.  
       But if $\pb\eps^\frac{\as}{2}$ is 
    bounded this shows that 
    \be\label{pe}\pb\eps^\frac{\as}{2}\to 0.\ee 
    
    We want to prove that 
    $$\int\uss(\pbq+2\pb\et+\etq)=o(1).$$  For the first term on the 
    left hand side we have, by 
    (\ref{pe}),
    $$
    \int\uss\pbq\leq C\pbq\eps^2\to 0.
    $$  
    For the third term we have
    $$
    \int\uss\etq\leq C|\et|_{\ts}^2\to 0.
    $$
       We claim that the remaining term also converges to zero.  This 
    will prove the lemma.  For the second term we have the estimate
    $$
    \zeta_{\eps}:=
    \left|\int\uss\pb\et\right|\leq|\et|_{\ts}|\pb|\left(\int
    U_{\eps,y_{\eps}}^{\frac 
    N{N-2}\frac{8}{N+2}}\right)^{\frac{N+2}{2N}}.
    $$
    If $N=5$, then 
    $$
    \zeta_{\eps}\leq C|\et|_{\ts}|\pb|\eps^{N\left(1-{\frac 
    {4}{N+2}}\right){\frac{N+2}{2N}}}\leq 
    C|\et|_{\ts}|\pb|\eps^{\frac{N-2}{2}}=
    C|\et|_{\ts}|\pb|\eps^{\frac 32}.
    $$
    If $N=6$, then 
    $$
    \zeta_{\eps}\leq C|\et|_{\ts}|\pb|\eps^2|\log\eps|^\frac{2}{3}.
    $$
    Finally, if $N\geq 7$, then 
    $$\zeta_{\eps}\leq C|\et|_{\ts}|\pb|\eps^2.$$ 
    In all three cases, (\ref{pe}) implies that $\zeta_{\eps}\to 0$.
    \end{proof}
\begin{rem}\label{td}
    If in the previous lemma, instead of assuming
    $\int|\nabla\phi_{\eps}|^2\to 0$, we assume that
    $\int|\nabla\phi_{\eps}|^2$ is bounded, then we can still conclude
    $\bar{\phi}_{\eps}{\eps^\frac{\as}{2}}\to 0$,
    $\int U_{\eps,y_{\eps}}^{\ts-2}\pbq\to 0$
    and
    $\int U_{\eps,y_{\eps}}^{\ts-2}\pb\eta_{\eps}
    \to 0$,  
    as $\eps\to 0$.
\end{rem}

{\noindent \bf Proof of Lemma~\ref{ttt}}  We 
basically adapt the argument of
the proof of Lem\-ma~3.3 of [APY]
(and Lemma 5.8 of Z.Q.\ Wang in
\cite{W1}), modified 
according to 
Lemma~\ref{tn}
and
Remark~\ref{td}.  The value of $k_{0}$ 
in Lemma~3.3 of [APY] 
is equal to $N$.

The proof is by induction.  We first consider $i=1$.
By the Rayleigh quotient, $\mu_{1,\eps}$ is given by
\bas
\mu_{1,\eps}&=&\inf\left\{\left.
|\nabla u|_{2}^2+\nu_{\eps}\int
U_{\eps,y_{\eps}}^{q\bt-2} u^2
    \right|
    \int U_{\eps,y_{\eps}}^\tsd u^2=1
\right\}\\
&=&\inf\left\{\left.\int_{\Om_{\eps}}
|\nabla v|^2+\eps^{2-\as}\nu_{\eps}\int_{\Om_{\eps}}
U^{q\bt-2} v^2
    \right|
    \int_{\Om_{\eps}} U^\tsd v^2=1
\right\}.
\eas
To estimate $\mu_{1,\eps}$ from above, we choose
$v_{\eps}:\Om_{\eps}\to\Rb$ defined by
$$
\txt
v_{\eps}
=U\left/\lb\int_{\Om_{\eps}}U^\ts\rb^{\frac{1}{2}}\right..$$
From the assumption $\eps^{2-\as}\nu_{\eps}\to 0$ as $\eps\to 0$, we get
$$
\mu_{1,\eps}\leq
\int_{\Om_{\eps}}
|\nabla v_{\eps}|^2+\eps^{2-s}\nu_{\eps}\int_{\Om_{\eps}}
U^{q\bt-2} v_{\eps}^2\to\frac
{\int_{\ttt}|\nabla U|^2}{\int_{\ttt}U^\ts}=\mu_{1}=1,
$$
as $\eps\to 0$.  Hence $\limsup_{\eps\to 0}\mu_{1,\eps}\leq\mu_{1}$.
Up to a subsequence, which 
we still denote by $\eps$,
$$\lim_{\eps\to 0}\mu_{1,\eps}=\hat\mu_{1}\leq\mu_{1}.$$

The functions $\phi_{1,\eps}$ satisfy
$$
\mu_{1,\eps}=|\nabla \phi_{1,\eps}|_{2}^2+\nu_{\eps}\int
U_{\eps,y_{\eps}}^{q\bt-2} \phi_{1,\eps}^2,
$$
so
$|\nabla \phi_{1,\eps}|_{2}$ is bounded and
$\int U_{\eps,y_{\eps}}^{q\bt-2} \phi_{1,\eps}^2\to 0$,
as $\eps\to 0$.  
If $\hat\mu_{1}$ were equal to 0, then $|\nabla \phi_{1,\eps}|_{2}\to 0$ 
and $\int U_{\eps,y_{\eps}}^{q\bt-2} \phi_{1,\eps}^2\to 0$,
as $\eps\to 0$. Lemma~\ref{tn} would imply
$1=\int U_{\eps,y_{\eps}}^{\ts-2} \phi_{1,\eps}^2\to 0$, 
as $\eps\to 0$, a 
contradiction.  So $\hat\mu_{1}\neq 0$ and 
$|\nabla \phi_{1,\eps}|_{2}\not\to 0$ as $\eps\to 0$.

The functions $\tilde{\phi}_{1,\eps}$ satisfy
    $$\left\{
    \begin{array}{ll}
    -\Delta \tilde{\phi}_{1,\eps}
    +\eps^{2-\as}\nu_{\eps}U^{q\bt-2}\tilde{\phi}_{1,\eps}
    =\mu_{1,\eps} U^\tsd\tilde{\phi}_{1,\eps}&\mbox{in\ 
    }\Om_{\eps},\\
    \tilde{\phi}_{1,\eps}>0&\mbox{in\ 
    }\Om_{\eps},\\
    \frac{\partial\tilde{\phi}_{1,\eps}}{\partial\nu}=0&\mbox{on\ 
    }\partial\Om_{\eps},\\
    \int_{\ttt}U^\tsd\tilde{\phi}_{1,\eps}^2=1.&
    \end{array}\right.
    $$
By the hypothesis $\lim_{\eps\to 0}\eps^{2-\as}\nu_{\eps}\to 0$,
and
elliptic regularity theory~\cite{ADN},
$\tilde{\phi}_{1,\eps}\to\hat{\phi}_{1}$ in 
$C^2_{\mbox{\tiny loc}}\lb\ttt\rb$, as $\eps\to 0$,
where $\hat\phi_{1}$ satisfies (\ref{bie}) with $\mu=\hat\mu_{1}$.
We conclude that $\hat\mu_{1}=\mu_{1}$ and 
$\hat\phi_{1}=\tilde{\phi}_{1}$.

We will now prove (\ref{tstdiff}) in case $i=1$, i.e.\
$$
    \int_{\Om_{\eps}}U^\tsd(\tilde{\phi}_{1,\eps}
    -\tilde{\phi}_{1})^2=
    \int U_{\eps,y_{\eps}}^\tsd({\phi}_{1,\eps}
    -{\fhi}_{1,\eps})^2\to
    0,\mbox{\ as\ }\eps\to 0,
$$
where
$\fhi_{1,\eps}(\ \cdot\ )=\eps^{-\frac{N-2}{2}}\tilde{\phi}_{1}\lb
\frac{\ \cdot\ -y_{\eps}}{\eps}\rb$.  The function $\tilde{\phi}_{1}$
belongs to $L^\ts(\Rb^{N})$ and $\fhi_{1,\eps}\weak 0$ in $H^1(\Om)$.
We denote the averages of $\phi_{1,\eps}$ 
and $\fhi_{1,\eps}$, in $\Om$, by
$\bar\phi_{1,\eps}$ and 
$\bar\fhi_{1,\eps}$, respectively.  By Poincar\'{e}'s inequality,
we can write
$\phi_{1,\eps}=\bar\phi_{1,\eps}+\eta_{1,\eps}$ and
$\fhi_{1,\eps}=\bar\fhi_{1,\eps}+\zeta_{1,\eps}$ 
with $|\eta_{1,\eps}|_{\ts}$ and $|\zeta_{1,\eps}|_{\ts}$ uniformly
bounded, as $\eps\to 0$.
Moreover,
\bas
\int U_{\eps,y_{\eps}}^\tsd({\phi}_{1,\eps}-{\fhi}_{1,\eps})^2&=&
\int U_{\eps,y_{\eps}}^\tsd({\bar\phi}_{1,\eps}-
{\bar\fhi}_{1,\eps})^2\\
&&+\ 2\int U_{\eps,y_{\eps}}^\tsd({\bar\phi}_{1,\eps}-
{\bar\fhi}_{1,\eps})({\eta}_{1,\eps}-{\zeta}_{1,\eps})
\\
&&+\ \int U_{\eps,y_{\eps}}^\tsd({\eta}_{1,\eps}-{\zeta}_{1,\eps})^2.
\eas
The first two terms on the right hand side converge to 0 as $\eps
\to 0$, due to Remark~\ref{td}.  But
\bas
\tilde{\eta}_{1,\eps}&=&\tilde{\phi}_{1,\eps}-
                         \tilde{\bar{\phi}}_{1,\eps}\\
                         &=&\tilde{\phi}_{1,\eps}-
                         \eps^{\frac{N-2}{2}}{\bar{\phi}}_{1,\eps}\\
                         &=&\tilde{\phi}_{1,\eps}-
                         \eps^{\frac{N-2-\as}{2}}
                         \lb\eps^{\frac{\as}{2}}
                         {\bar{\phi}}_{1,\eps}\rb\\
\noalign{\noindent\mbox{and}}
\tilde{\zeta}_{1,\eps}&=&\tilde{\phi}_{1}-
                         \eps^{\frac{N-2}{2}}
                         {\bar{\fhi}}_{1,\eps}.
\eas
These equalities and, again,
Remark~\ref{td} show that $\tilde{\eta}_{1,\eps}\to
\tilde{\zeta}_{1,\eps}$ in $C^2_{\mbox{\tiny loc}}\lb\ttt\rb$,
as $\eps\to 0$. We
conclude that the term
$$
\int_{\Om_{\eps}}
U^\tsd\lb{\tilde\eta}_{1,\eps}-{\tilde\zeta}_{1,\eps}\rb^2
$$
also converges to 0  as $\eps\to 0$
(see (3.42) and (3.43) in the proof of Lemma~3.3 of [APY]).
This proves (\ref{tstdiff}) for $i=1$.

Now assume that (\ref{muimu}) and
(\ref{tstdiff}) hold for $1\leq i\leq L-1$,

To estimate $\mu_{L,\eps}$ from above we choose
$v_{\eps}:\Om_{\eps}\to\Rb$ defined by
$v_{\eps}=\tilde\phi_{L}$.  Let 
$v_{\eps}=\sum_{i=1}^\infty a_{i,\eps}\tilde{\phi}_{i,\eps}$.
Clearly,
\be\label{aieu}
\sum_{i=1}^\infty \mu_{i,\eps}a_{i,\eps}^2=
\int_{\Om_{\eps}}
|\nabla v_{\eps}|^2+\eps^{2-\as}\nu_{\eps}\int_{\Om_{\eps}}
U^{q\bt-2} v_{\eps}^2\to\mu_{L}
\ee
and
\be\label{aied}
\sum_{i=1}^{L-1} \mu_{i,\eps}a_{i,\eps}^2+
\mu_{L,\eps}\sum_{i=L}^\infty a_{i,\eps}^2\leq
\sum_{i=1}^\infty \mu_{i,\eps}a_{i,\eps}^2.
\ee

We claim that $a_{i,\eps}\to 0$ for $1\leq i\leq L-1$
and $\sum_{i=L}^\infty a_{i,\eps}^2\to 1$, as $\eps\to 0$.
Indeed, 
\bas
a_{i,\eps}&=&\int_{\Om_{\eps}}U^\tsd v_{\eps}\tilde{\phi}_{i,\eps}\\
&=&
\int_{\Om_{\eps}}U^\tsd v_{\eps}\tilde{\phi}_{i}+
\int_{\Om_{\eps}}U^\tsd v_{\eps}(
\tilde{\phi}_{i,\eps}-\tilde{\phi}_{i}).
\eas
As $\eps\to 0$,
for $1\leq i\leq L-1$, 
the first term on the right hand side approaches
$\int_{\ttt}U^\tsd\tilde{\phi}_{L}\tilde{\phi}_{i}=0$ 
and the second one
is bounded by
$$
\lb\int_{\Om_{\eps}}U^\tsd v_{\eps}^2\rb^\frac{1}{2}
\lb\int_{\Om_{\eps}}U^\tsd (
\tilde{\phi}_{i,\eps}-\tilde{\phi}_{i})^2\rb^\frac{1}{2}\to 0.
$$
Moreover,
$$
\sum_{i=1}^\infty a_{i,\eps}^2=\int_{\Om_{\eps}}U^\tsd v_{\eps}^2
\to\int_{\ttt}U^\tsd \tilde{\phi}_{L}^2=1.
$$
This proves our claim.

Combining (\ref{aieu}), (\ref{aied}) and the previous claim,
we have $\limsup_{\eps\to 0} \mu_{L,\eps}\leq\mu_{L}$.
Up to a subsequence, which we still denote by $\eps$,
$$
\lim_{\eps\to 0}\mu_{L,\eps}=\hat\mu_{L}\leq\mu_{L}.
$$

The value of $\mu_{L,\eps}$ is
\bas
\mu_{L,\eps}&=&\inf\left\{\left.\int_{\Om_{\eps}}
|\nabla v|^2+\eps^{2-\as}\nu_{\eps}\int_{\Om_{\eps}}
U^{q\bt-2} v^2
    \right|
    \int_{\Om_{\eps}} U^\tsd v^2=1,\right.\\ 
   &&\qquad \left.\int_{\Om_{\eps}}U^\tsd\tilde{\phi}_{i,\eps}v=0
    \mbox{\ for\ }1\leq i\leq L-1
\right\}.
\eas
We repeat part of the argument given for $i=1$ and conclude
that $\tilde\phi_{L,\eps}\to\hat\phi_{L}$ in $C^2_{\mbox{\tiny loc}}
\lb\ttt\rb$, as $\eps\to 0$, where $\hat\phi_{L}$
satisfies (\ref{bie}) with $\mu=\hat\mu_{L}$.
The space  
consisting
in the completion with norm 
$|\nabla\ \cdot\ |_{L^2(\Rb^{N})}$ of the  smooth functions
with compact support in $\Rb^{N}$ is dense in $L^2(U^{\ts-2}dx)$.
Therefore 
the Gagliardo-Nirenberg-Sobolev inequality implies that 
$\hat\phi_{L}\in L^\ts(\Rb^{N})$.
So we can also conclude
\be
\lim_{\eps\to 0}\int_{\Om_{\eps}}U^\tsd(\tilde\phi_{L,\eps}
-\hat\phi_{L})^2=0.\label{chapeu}
\ee

To prove that $\hat\phi_{L}=\tilde\phi_{L}$ and 
$\hat\mu_{L}=\mu_{L}$, we show that
$\hat\phi_{L}$ is orthogonal to $\tilde\phi_{i}$ for
$1\leq i\leq L-1$. So let $1\leq i\leq L-1$.  Then
$$
0=\int U_{\eps,y_{\eps}}^\tsd\phi_{i,\eps}\phi_{L,\eps}=
\int_{\Om_{\eps}} U^\tsd\tilde\phi_{i,\eps}\tilde\phi_{L,\eps}
\to
\int_{\ttt} U^\tsd\tilde\phi_{i}\hat\phi_{L},
$$
as $\eps\to 0$, as the difference
$
\int_{\Om_{\eps}} U^\tsd\tilde\phi_{i,\eps}\tilde\phi_{L,\eps}
-
\int_{\ttt} U^\tsd\tilde\phi_{i}\hat\phi_{L},
$
approaches zero, as $\eps\to 0$, because of (\ref{tstdiff})
for $i=i$ and (\ref{chapeu}).  Indeed,
\bas
\int_{\Om_{\eps}} U^\tsd\tilde\phi_{i,\eps}\tilde\phi_{L,\eps}
-
\int_{\ttt} U^\tsd\tilde\phi_{i}\hat\phi_{L}&=&
\int_{\Om_{\eps}} U^\tsd\tilde\phi_{i}\hat\phi_{L}
-
\int_{\ttt} U^\tsd\tilde\phi_{i}\hat\phi_{L}
\\
&&+\
\int_{\Om_{\eps}} U^\tsd
(\tilde\phi_{i,\eps}-\tilde{\phi}_{i})\hat\phi_{L}
\\
&&+\ 
\int_{\Om_{\eps}} U^\tsd
\tilde\phi_{i,\eps}(\tilde\phi_{L,\eps}-\hat\phi_{L}).
\eas
\hfill$\Box$

Using Lemma~\ref{ttt} and the proof of Lemma 3.4 of [APY],
we deduce
\begin{lem}\label{dif}
    Suppose $y_{\eps}\in\partial\Om$, $\lim_{\eps\to 0}y_{\eps}=y_{0}$
    and
    $\lim_{\eps\to 0}\eps^{2-\as}\nu_{\eps}=0$.
    There 
    exists a constant $\gamma_{1}>0$ such that, for sufficiently 
    small $\eps$,
    $$|\nabla w|_{2}^2+\nu_{\eps}
    \int U_{\eps,y_{\eps}}^{q\bt-2} w^2
    \geq(\tsu+\gamma_{1})
    \int U_{\eps,y_{\eps}}^\tsd w^2+O(\eps^2||w||^2)
    $$
    for $w$ orthogonal to $T_{1,\eps,y_{\eps}}({\mathcal{M}})$.
\end{lem}

\section{Proof of the main theorem}\label{four}

In this section we prove Theorem~\ref{theorem} and
give one more lower bound for $\alpha_{0}$,
in addition to one in Remark~\ref{lbum}.

Assume the positive functions
$$\uk=C_{k}U_{\eps_{k},y_{k}}+w_{k},$$ 
satisfy 
(\ref{sndd}),
(\ref{m}), (\ref{ck}), (\ref{akek}) and (\ref{zero}).

We start by collecting some useful estimates.
For brevity, we shall write $$U_{k}:=U_{\eps_{k},y_{k}}.$$

{\em Estimate for $\iukw$:\/}  From Lemma 4.1 of \cite{G},
\be
    \left|
    \iukw
    \right|\leq\left\{
    \begin{array}{ll}
        O\left(\eps_k^\frac 32\nwk\right)&\mbox{if\ }N=5,\\
        O\left(\eps_k^2|\log\ek|^\frac 23\nwk\right)&\mbox{if\ }N=6,\\
        O\left(\eps_k^2\nwk\right)&\mbox{if\ }N\geq 7.
    \end{array}
    \right.\label{uwi}
\ee

{\em Estimate for $\iusu$:\/} From [APY], Equations (3.15), for $N\geq 
5$,
\be\label{iusu}
\iusu=O(\ek\nwk).\ee

{\em Estimate for $\int U_{k}^{q-1}|\wk|$:\/} Since 
$(q-1)\frac{2N}{N+2}\geq\lb\frac{2N}{N-1}-1\rb\frac{2N}{N+2}=
\frac{N+1}{N-1}\frac{2N}{N+2}>\frac{N}{N-2}$ (for $N\geq 4$),
\ba\label{iuzu}
\int U_{k}^{q-1}|\wk|&\leq&\nwks\left(\int 
U_{k}^{(q-1)\frac{2N}{N+2}}\right)^\frac{N+2}{2N}\nonumber\\
&\leq&C\nwks\eps_{k}^{N\left(1-{\frac{q-1}{\ts-1}}
\right){\frac{\ts-1}{\ts}}}\nonumber\\
&=&C\nwks\eps_{k}^{\frac{N-2}{2}(\ts-q)}
\nonumber\\
&=&O(\eps_{k}^{{1}/{\bt}}\nwk).
\ea

{\em Estimate for $\int U_{k}^{q\bt-2}w_{k}^2$:\/} 
If $\tf\leq q<\tz$, then $(q\bt-2)\frac{N}{2}<\frac{N}{N-2}$.
From (\ref{iuzdi}) in the proof of Lemma~\ref{tn},
\be\label{iuqtd}
\int U_{k}^{q\bt-2}|\wk|^2\leq C|\wk|_{\ts}^2\eps_{k}^\as.
\ee
If $q=\tz$, from (\ref{iuzd}) in the proof of Lemma~\ref{tn},
\be\label{iuqtdz}
\int U_{k}^{q\bt-2}|\wk|^2\leq
C|w_{k}|_{\ts}^2\eps_{k}|\log\eps_{k}|^\frac 2N.
\ee

{\em Estimate for $\iusd$:\/} 
\be\label{iusd}
\iusd=O(\nwkq).
\ee

Now we will obtain a lower bound for $\Psi_{\ak}(\uk)$.
Let $v_{k}=\uk/\ck=
U_{k}+\tilde{w}_{k}=U_{k}+\wk/\ck$.  Because of (\ref{ck}), 
the sequence $(v_{k})$ 
satisfies (\ref{sndd}) and the sequence $\tilde{w}_{k}$ 
satisfies (\ref{zero}).
Of course, $d(v_{k},M)$ is achieved by $\Uk$.
Because $\Psi_{\alpha}$ is 
homogeneous of degree zero, $\Psi_{\ak}(\uk)=\Psi_{\ak}(v_{k})$.  
We will compute 
$\Psi_{\ak}(v_{k})$ but we will still call $v_{k}$ by $\uk$, and 
$\tilde{w}_{k}$
by $\wk$. 

The value of $\Psi_{\alpha_k}(u_{k})$ is the sum of 
$\beta(u_{k})$ and $\ak\beta(u_{k})\delta(u_{k})$. 
As in \cite{G}, we can obtain the following lower bound for 
$\beta(u_{k})$:
$$\beta(u_{k})
    \geq\frac{\ngukq}{\nuksq}+
    2^{\frac{N-2}N}S^{\frac{2-N}{2}}\left[
    \gamma_{2}\pnhwkq
        -\ (\tsu)\iusd
    \right]+o(\ek),
$$
for any fixed number $\gamma_{2}<1$. 

We also wish to obtain a lower bound for 
\be\label{prod}\alpha_{k}\beta(u_{k})\delta(u_{k})=\ak
\frac{\pnhukpqmeio^\as}{|u_{k}|_{\ts}^{2+\ts\as/2}}\uqb.
\ee

We obtain a lower bound for $\pnhukpqmeio^\as$ from 
$$
    \nhukpq=\pnhukq+2{\txt\left(\igukw+a\iukw\right)} 
    +\ \pnhwkq.
$$
Using 
(2.17) and (2.38) in Adimurthi and Mancini \cite{AM},
(\ref{dot}) and (\ref{uwi}),
$$
\nhukpq=\sndd+O(\ek)+O(\nwkq).
$$
This implies that
$$
||\uk||^\as\geq \left(\sndd\right)^\frac{\as}{2}+O(\ek)+O(\nwkq).
$$

We obtain a lower bound for $|u_{k}|_{\ts}^{-(2+\ts\as/2)}$ from 
$$
\nukpss=\iuks+\ts\iusu+{\txt\frac{\ts(\tsu)}{2}}\iusd+O(\nwkr)
$$
(see [APY]),
where $r=\min\{\ts,3\}$, i.e., $r=3$ if $N=5$, and $r=\ts$ if $N>5$.
Using  (2.18) in Adimurthi and Mancini \cite{AM},
(\ref{iusu}), (\ref{iusd}) and 
$$
(1+z)^{-\eta}\geq 1-\eta z,
$$
for $\eta>0$ and $z\geq -1$, we deduce
$$
|u_{k}|_{\ts}^{-(2+\ts\as/2)}
\geq\left(\sndd\right)^{-\frac{\as}{2}-\frac{2}{\ts}}+
O(\ek)+O(\nwkq).
$$

For the product we obtain the lower bound
\ba
\frac{\pnhukpqmeio^\as}{|u_{k}|_{\ts}^{2+\ts\as/2}}&\geq&
2^{\frac{N-2}{N}}S^{\frac{2-N}{2}}+O(\ek)+O(\nwkq)\label{fum}\\
&=&\du+\ddois+\dt.\nonumber
\ea

To estimate $|u_{k}|_{q}^{q\bt}$ we use 
\begin{lem}\label{calculus}  Suppose $\tf\leq q\leq\tz$ and
    $\bt$ is given by\/ {\rm (\ref{bt})}.
    For $x\geq-1$,
    $$
    (1+x)^q\geq 1+\frac{q\bt}{2}|x|^\frac{2}{\bt}-q|x|.
    $$
\end{lem}
\begin{proof}
    From (\ref{qba})
    if follows that $\frac{2}{\bt}\leq q$, as $\as\geq 0$.
    We will consider separately the cases $x>0$ and $x<0$, since the 
    inequality is obviously true if $x=0$.
    For $x\geq-1$, $x\neq 0$, define
    $$
    f(x)=(1+x)^q-1-\frac{q\bt}{2}|x|^\frac{2}{\bt}+q|x|.
    $$
    Then
    $$
    f'(x)=q(1+x)^{q-1}-q|x|^{\frac{2}{\bt}-1}\mbox{sign\,}\/ x
    +q\,\mbox{sign\,}\/ x.
    $$
    If $x>0$, then $f'(x)>q(1+x)^{q-1}-qx^{\frac{2}{\bt}-1}>0$.
    If $-1\leq x<0$, then
    $$
    f'(x)=q(1+x)^{q-1}+qx^{\frac{2}{\bt}-1}-q\leq 0,
    $$
    as
    $$
    (1+x)^{q-1}+x^{\frac{2}{\bt}-1}\leq 1,
    $$
    since both $x\mapsto(1+x)^{q-1}$ and
    $x\mapsto x^{\frac{2}{\bt}-1}$ are convex.
\end{proof}

As a consequence of Lemma~\ref{calculus},
$$
    |\uk|_{q}^{q\bt}\geq\lb|U_{k}|_{q}^q+{\txt\frac{q\bt}{2}}
    \int U_{k}^{q-\frac{2}{\bt}}|\wk|^\frac{2}{\bt}
    -\eta_{k}\rb^\bt,
$$
with
$$
\eta_{k}:=\min\left\{q\int U_{k}^{q-1}|\wk|,\
|U_{k}|_{q}^q+{\txt\frac{q\bt}{2}}
    \int U_{k}^{q-\frac{2}{\bt}}|\wk|^\frac{2}{\bt}
\right\}
$$
We now use the fact that $(1-x)^\eta\geq 1-x$ for $0\leq x\leq 1$ and
$0<\eta\leq 1$ to write
$$
    |\uk|_{q}^{q\bt}\geq\lb|U_{k}|_{q}^q+{\txt\frac{q\bt}{2}}
    \int U_{k}^{q-\frac{2}{\bt}}|\wk|^\frac{2}{\bt}\rb^\bt
    -q|U_{k}|_{q}^{q(\bt-1)}\int U_{k}^{q-1}|\wk|.
$$
Estimates (\ref{iuzu}), (\ref{iukz}) and the H\"older inequality 
yield
\bas
    |\uk|_{q}^{q\bt}&\geq&\lb|U_{k}|_{q}^q+{\txt\frac{q\bt}{2}}
    \int U_{k}^{q-\frac{2}{\bt}}|\wk|^\frac{2}{\bt}\rb^\bt
    -C\eps_{k}^{(\bt-1)\frac{1}{\bt}}\eps_{k}^{{1}/{\bt}}||\wk||\\
    &=&\lb|U_{k}|_{q}^q+{\txt\frac{q\bt}{2}}
    \int U_{k}^{q-\frac{2}{\bt}}|\wk|^\frac{2}{\bt}\rb^\bt
    +O(\ek)||\wk||\\
    &=&{\txt\frac{1}{2^{1-\bt}}}|U_{k}|_{q}^{q\bt}+
    {\txt\frac{1}{2^{1-\bt}}}
    {\txt\frac{q^\bt\bt^\bt}{2^\bt}}\lb
    \int U_{k}^{q-\frac{2}{\bt}}|\wk|^\frac{2}{\bt}\rb^\bt
    +O(\ek)||\wk||\\
    &\geq&{\txt\frac{1}{2^{1-\bt}}}|U_{k}|_{q}^{q\bt}+
    {\txt\frac{1}{|\Om|^{1-\bt}}}
    {\txt\frac{q^\bt\bt^\bt}{2}}
    \int U_{k}^{q\bt-2}w_{k}^2
    +O(\ek)||\wk||.
\eas
Using (\ref{iukz}) again,
\ba\label{F}
\ak|\uk|_{q}^{q\bt}&\geq&{\txt\frac{B(q,N)^\bt}{2}}\ak\ek+
{\txt\frac{1}{|\Om|^{1-\bt}}}
    {\txt\frac{q^\bt\bt^\bt}{2}}\ak
    \int U_{k}^{q\bt-2}w_{k}^2
+o(\ak\ek)\\
&=&\fu+\fd+\ft.\nonumber
\ea

The next step is to substitute (\ref{fum}) and (\ref{F}) in (\ref{prod}).
We notice that 
$$(\du+\ddois+\dt)\ft=o(\ak\ek)$$ 
and
$$
(\ddois+\dt)\fu=o(\ak\ek).
$$
The term $\ddois\fd$ is also $o(\ak\ek)$.  In fact, 
if $\tf\leq q<\tz$, by (\ref{iuqtd}), 
$$
\ddois\fd=O\lb\ak\eps_{k}^{\as+1}\nwkq\rb=o(\ak\ek).
$$
If $q=\tz$, by (\ref{iuqtdz}),
$$
\ddois\fd=O\left(\ak\eps_{k}^2|\log\ek|^{\frac{2}{N}}
\nwkq\right)=o(\ak\ek).
$$
So,
\bas
\ak\beta(u_{k})\delta(u_{k})&\geq&
2^{\frac{N-2}{N}}S^{\frac{2-N}{2}}
{\txt\frac{B(q,N)^\bt}{2}}\ak\ek\\
&&+\
2^{\frac{N-2}{N}}S^{\frac{2-N}{2}}
{\txt
\frac{1}{|\Om|^{1-\bt}}\frac{q^\bt\bt^\bt}{2}
}
\ak\int U_{k}^{q\bt-2}w_{k}^2\\
&&+\ O\left(\nwkq\right)\ak\int U_{k}^{q\bt-2}w_{k}^2+o(\ak\ek)\\
&\geq&
2^{-\frac{2}{N}}
S^{\frac{2-N}{2}}\left[B(q,N)^\bt\ak\ek+
\gamma_{2}
{\txt
\frac{1}{|\Om|^{1-\bt}}q^\bt\bt^\bt
}
\ak\int U_{k}^{q\bt-2}w_{k}^2\right]\\
&&+\ o(\ak\ek),
\eas
for any fixed number $\gamma_{2}<1$. 
This is our lower bound for $\ak\beta(u_{k})\delta(u_{k})$.

Combining the lower bounds for $\beta(u_{k})$ and for 
$\ak\beta(u_{k})\delta(u_{k})$,
\bas
    \Psi_{\ak}(\uk)&\geq&
    \frac{\ngukq}{\nuksq}+2^{-\frac{2}{N}}
    S^{\frac{2-N}{2}}B(q,N)^\bt\ak\ek\\
    &&+\ 
    2^{\frac{N-2}N}S^{\frac{2-N}{2}}\left[\gamma_{2}
    \pnhwkq+\gamma_{2}
    {\txt
    \frac{1}{|\Om|^{1-\bt}}\frac{q^\bt\bt^\bt}{2}
    }
    \ak\int U_{k}^{q\bt-2}w_{k}^2
    \right.\\
    &&\left.\qquad -\ (\tsu)\iusd\right]+o(\ak\ek).
\eas
From Lemma~\ref{dif}, the term inside the square parenthesis is greater 
than
$$
    \left[\left(
    \gamma_{2}-\frac{(\tsu)}{(\tsu)+\gamma_{1}}\right)\left(\pnhwkq
    +
    {\txt
    \frac{1}{|\Om|^{1-\bt}}\frac{q^\bt\bt^\bt}{2}
    }
    \ak\int U_{k}^{q\bt-2}w_{k}^2\right)+o(\ek)
    \right].
$$
We choose $\gamma_{2}\geq\frac{(\tsu)}{(\tsu)+\gamma_{1}}$. 
As a consequence, this term is greater than 
$o(\ek)$.
Hence,
$$
\Psi_{\ak}(\uk)\geq\frac{\ngukq}{\nuksq}+2^{-\frac{2}N}S^\frac{2-N}{2}
B(q,N)^\bt\ak\ek+o(\ak\ek).
$$
Recall, from Adimurthi and 
Mancini \cite{AM}, that for $N\geq 5$ and $y\in\partial\Om$,
\be\label{ei}
\frac{|\nabla U_{\eps,y}|_{2}^2}{|U_{\eps,y}|_{\ts}^2}=
\frac{S}{2^\frac 
2N}-2^{\frac{N-2}N}SA(N)H(y)
\eps+O(\eps^2),
\ee
with
$$
A(N)=
{\textstyle
\frac{N-1}{N}}\frac{1}{\sqrt{\pi}}\,
\frac{\Gamma\left(\frac{N-3}{2}\right)}{\Gamma\left(\frac{N-2}{2}\right)}
$$
and $H(y)$ the mean curvature
of $\partial\Om$  at $y$ 
with respect to the unit outward normal.
Therefore,
\bas
\Psi_{\ak}(\uk)&
\geq&
\frac{S}{2^\frac 2N}\\ &&
+\ 2^{-\frac{2}N}S^\frac{2-N}{2}B(q,N)^\bt\ak\ek\left[1-2
S^{\frac N2}\frac{A(N)}{B(q,N)^\bt}H(y_{k})\frac 1\ak+o(1)
\right]\\
&>&\frac{S}{2^\frac 2N},
\eas
for large $k$.
\begin{rem}
    If in the argument above, instead of using the inequality
    $$\txt(x+y)^\bt\geq\frac{1}{2^{1-\bt}}x^\bt+
    \frac{1}{2^{1-\bt}}y^\bt,$$
    we use
    $$\txt(x+y)^\bt\geq(1-\varsigma)^{1-\bt}x^\bt+
    \varsigma^{1-\bt}y^\bt,$$ for $x$, $y>0$ and
    $\varsigma$ such that $0<\varsigma<1$,
    then we obtain the following lower bound for 
$\Psi_{\ak}(\uk)$:
$$
\frac{S}{2^\frac 2N}
+{\txt\frac{2^{\frac{N-2}{N}}}{S^\frac{N-2}{2}}}
{\txt\frac{(1-\varsigma)^{1-\bt}B(q,N)^\bt}{2^\bt}}\ak\ek\left[1-
{\txt\frac{2^\bt}{(1-\varsigma)^{1-\bt}}}S^{\frac N2}
{\txt\frac{A(N)}{B(q,N)^\bt}}
{\txt\frac{H(y_{k})}{\ak}}+o(1)
\right].$$
\end{rem}
{\bf Proof of Theorem~\ref{theorem}.}
So assume $\alpha_{0}$, in (\ref{min}),  is $+\infty$.  Choose a
    sequence $\alpha_{k}\to+\infty$ as $k\to+\infty$ and denote by $\uk$ 
    a positive minimizer for $\Psi_{\alpha_{k}}$ satisfying \rfk.  From 
    Lemmas~\ref{infinito} and \ref{instantao}, the conditions 
    (\ref{sndd}), (\ref{m}), 
    (\ref{ck}), (\ref{akek}) and (\ref{zero}) hold.  Hence 
    $S_{\ak}=\Psi_{\ak}(\uk)>\sddn$ for large $k$, which is impossible.
    Therefore $\alpha_{0}$ is finite.  
    Remarks~\ref{pscondition} and \ref{coro} imply
    Theorem~\ref{theorem}.
\hfill$\Box$
        
We give one more lower bound for $\alpha_{0}$,
in addition to one in Remark~\ref{lbum}.
\begin{lem}\label{lbdois}
   The value $\alpha_{0}$ has the lower bound
    $$\alpha_{0}\geq 2^\bt S^\frac{N}{2}\frac{A(N)}{B(q,N)^\bt}
    \max_{\partial\Om}H.$$
\end{lem}
\begin{proof}
    Suppose $\alpha<2^\bt S^\frac{N}{2}\frac{A(N)}{B(q,N)^\bt}
    \max_{\partial\Om}H$.
    Choose $P\in\partial\Om$ such that $H(P)=\max_{\partial\Om}H$. 
    From 
    (\ref{ei}),
    $$\beta(U_{\eps,P})=\frac{S}{2^\frac 
2N}-2^{\frac{N-2}N}S
A(N)H(P)\eps+o(\eps).$$
From 
(2.17) and (2.38) in Adimurthi and 
Mancini \cite{AM}, 
$$
||U_{\eps,P}||^\as\leq\lb\sndd\rb^\frac{\as}{2}+O(\eps)
$$
and, from (2.18) in Adimurthi and Mancini \cite{AM},
$$
|U_{\eps,P}|_{\ts}^{-(2+\ts\!\as/2)}\leq
\lb\sndd\rb^{-\lb\frac{\as}{2}+\frac{2}{\ts}\rb}+O(\eps).
$$
Together,
$$
\frac{||U_{\eps,P}||^\as}{|U_{\eps,P}|_{\ts}^{2+\ts\!\as/2}}
\leq 2^\frac{N-2}{N}S^\frac{2-N}{2}+O(\eps).
$$
Moreover, from (\ref{iukz}),
$$
|U_{\eps,P}|_{q}^{q\bt}\leq{\txt\frac{B(q,N)^\bt}{2^\bt}}\eps+
O(\eps^{2})
$$
Combining the last two estimates,
$$
(\beta\delta)(U_{\eps,P})\leq
2^\frac{N-2}{N}S^\frac{2-N}{2}
{\txt\frac{B(q,N)^\bt}{2^\bt}}\eps+
o(\eps).
$$
We can estimate $S_{\alpha}$ from above by
\bas
S_{\alpha}&\leq&\Psi_{\alpha}(U_{\eps,P})\\ &\leq&
\frac{S}{2^\frac 2N}
-2^{\frac{N-2}N}S^\frac{2-N}{2}
{\txt\frac{B(q,N)^\bt}{2^\bt}}\alpha\eps\left[2^\bt
S^{\frac N2}\frac{A(N)}{B(q,N)^\bt}H(P)\frac 1\alpha-1+o(1)
\right]
\eas
as $\eps\to 0$.  Since we are supposing 
$\alpha<2^\bt S^\frac{N}{2}\frac{A(N)}{B(q,N)^\bt}
    H(P)$, the value of $S_{\alpha}$ satisfies $S_{\alpha}<\sddn$.
    This proves the lemma.
\end{proof}

\section*{Appendix: The estimate for $|U_{\eps,y}|_{q}^q$ for
$y\in\partial\Om$}

In this Appendix we 
prove that if $y\in\partial\Omega$, then
\ba|U_{\eps,y}|_{q}^q
&=&\frac{B(q,N)}{2}\eps^{(\ts-q)\frac{N-2}{2}}+
O(\eps^{1+1/{\bt}}),\nonumber\\
&=&\frac{B(q,N)}{2}\eps^{1/{\bt}}+O(\eps^{1+{1}/{\bt}}),\label{iukz}
\ea
with
$$
B(q,N)=
\int_{\Rb^{N}}U^q=\pi^\frac{N}{2}[N(N-2)]^\frac{N}{2}
\frac{\Gamma\left(\frac{N-2}{2}q-\frac{N}{2}\right)
}{\Gamma\left(\frac{N-2}{2}q\right)},
$$
by adapting an estimate due to Adimurthi 
and Mancini~\cite{AM} ($U_{\eps,y}$ is defined in (\ref{resins})).
By a change of coordinates we can assume that $y=0$,
$$
B_{R}(0)\cap\Om=\{(x',x_{N})\in B_{R}(0)|x_{N}>\rho(x')\}
$$
and
$$
B_{R}(0)\cap\partial\Om=\{(x',x_{N})\in B_{R}(0)|x_{N}=\rho(x')\},
$$
for some $R>0$, where $x'=(x_{1},\ldots,x_{N-1})$,
$$\rho(x')=\sum_{i=1}^{N-1}\lambda_{i}x_{i}^2+O(|x'|^3),$$
$\lambda_{i}\in\Rb$, $1\leq i\leq N-1$.  

Let 
$U_{\eps}:=U_{\eps,0}$ and
$\Sigma:=\{(x',x_{N})\in B_{R}(0)|0<x_{N}<\rho(x')\}$.
Then
\be\label{sumudt}
|U_{\eps}|_{q}^q=
\frac 12\int_{B_{R}(0)}U_{\eps}^q-
\int_{\Sigma}U_{\eps}^q+\int_{B_{R}^{C}(0)\cap\Om}U_{\eps}^q
\ee
if all the $\lambda_{i}$'s are positive.
If all the $\lambda_{i}$'s are 
negative, then the minus sign on the right hand side 
turns into a plus sign.  
Henceforth we will assume all the $\lambda_{i}$'s are positive.
The final estimate for $|U_{\eps}|_{q}^q$ will hold no matter what the 
sign of the $\lambda_{i}$'s, for it holds when 
the $\lambda_{i}$'s all have the same sign.

We will estimate each of the three terms on the right hand side of
(\ref{sumudt}).  For the third term we have
\bas
\int_{B_{R}^{C}(0)\cap\Om}U_{\eps}^q&\leq&
\int_{B_{R}^{C}(0)}U_{\eps}^q\\
&=&O\left(\eps^\frac{1}{\bt}\int_{R/\eps}^{+\infty}
\frac{r^{N-1}}{(1+r^2)^{\frac{N-2}{2}q}}dr\right)\\
&=&O\lb\eps^\frac{1}{\bt}\times\eps^{N-\frac{2}{\bt}}\rb\\
&=&O\lb\eps^{N-\frac{1}{\bt}}\rb\\
&=&O\lb\eps^{\frac{N-2}{2}q}\rb.
\eas

Using this estimate, for the first term on the 
right hand side of (\ref{sumudt}) 
we have
\bas
\frac 12\int_{B_{R}(0)}U_{\eps}^q&=&\frac 
12\int_{\Rb^{N}}U_{\eps}^q+O\lb\eps^{N-\frac{1}{\bt}}\rb\\
&=&\frac{1}{2}\eps^\frac{1}{\bt}\int_{\Rb^{N}}
U^q+O\lb\eps^{N-\frac{1}{\bt}}\rb\\
&=&\frac{1}{2}B(q,N)\eps^\frac{1}{\bt} +O\lb\eps^{N-\frac{1}{\bt}}\rb,
\eas
with
\bas
B(q,N)&:=&\int_{\Rb^{N}}U^q\\
&=&[N(N-2)]^\frac{N}{2}\omega_{N}\int_{0}^{+\infty}
\frac{r^{N-1}}{(1+r^2)^{\frac{N-2}{2}q}}\,dr\\
&=&[N(N-2)]^\frac{N}{2}\omega_{N}
\frac{\Gamma\left(\frac{N}{2}\right)
\Gamma\left(\frac{N-2}{2}q-\frac{N}{2}\right)
}{2\Gamma\left(\frac{N-2}{2}q\right)}\\
&=&\pi^\frac{N}{2}[N(N-2)]^\frac{N}{2}
\frac{\Gamma\left(\frac{N-2}{2}q-\frac{N}{2}\right)
}{\Gamma\left(\frac{N-2}{2}q\right)};\\
\noalign{\noindent\mbox{in particular,}}\\
B(\tf,N)&=&\pi^\frac{N}{2}[N(N-2)]^\frac{N}{2}
\frac{\Gamma\left(\frac{N(N-3)}{2(N-1)}\right)
}{\Gamma\left(\frac{N(N-2)}{N-1}\right)}\\
\noalign{\noindent\mbox{and}}\\
B(\tz,N)&=&\pi^\frac{N}{2}[N(N-2)]^\frac{N}{2}
\frac{\Gamma\left(\frac{N-2}{2}\right)
}{\Gamma\left(N-1\right)}.
\eas

So we are left with the estimate of the second term on the right hand 
side of (\ref{sumudt}).  Let $\sigma>0$ be such that
$$
L_{\sigma}:=\{x\in\Rb^{N}|\,|x_{i}|<\sigma, 1\leq i\leq N\}\subset 
B_{\frac{R}{4}}(0)
$$
and define
$$
\Delta_{\sigma}:=\{x'|\,|x_{i}|<\sigma,1\leq i\leq N-1\}.
$$
For the second term on the right hand side of (\ref{sumudt}),
\bas
   \int_{\Sigma}U_{\eps}^q&=&\int_{\Sigma\cap 
   L_{\sigma}}U_{\eps}^q+O\lb\eps^{N-\frac{1}{\bt}}\rb\\
   &=&\int_{\Delta_{\sigma}}\int_{0}^{\rho(x')}U_{\eps}^q
   dx_{N}\,dx'+O\lb\eps^{N-\frac{1}{\bt}}\rb\\
   &=&O\left(\int_{\Delta_{\sigma}}\int_{0}^{\rho(x')}
   \frac{\eps^{\frac{N-2}{2}q}}{(\eps^2+|x|^2)^{\frac{N-2}{2}q}}\,
   dx_{N}\,dx'\right)+O\lb\eps^{N-\frac{1}{\bt}}\rb;\\
\noalign{\noindent using the change of variables 
$\sqrt{\eps^2+|x'|^2}\,y_{N}=x_{N}$,}
   &=&O\left(\int_{\Delta_{\sigma}}
   \frac{\eps^{\frac{N-2}{2}q}}{(\eps^2+|x'|^2)^{
   \frac{N-2}{2}q-{\frac{1}{2}}}}   
   \int_{0}^{\frac{\rho(x')}{\sqrt{\eps^2+|x'|^2}}}
   \frac{1}{(1+y_{N}^2)^{\frac{N-2}{2}q}}\,
   dy_{N}\,dx'\right)\\
   &&+\ O\lb\eps^{N-\frac{1}{\bt}}\rb;\\
\noalign{\noindent 
let $\kappa\geq 0$;
$\int_{0}^s\frac{1}{(1+t^2)^{\kappa}}\,dt\leq s$ for
$s>0$ and
$\int_{0}^s\frac{1}{(1+t^2)^{\kappa}}\,dt
=s-\frac{\kappa}{3}s^3+\frac{\kappa(\kappa+1)}{10}s^5-O(s^7)$
for small $s$; thus $\int_{0}^s\frac{1}{(1+t^2)^{\kappa}}\,dt
=s+O(s^3)$ for all $s$ and we can continue}
    &=&O\left(\eps^{\frac{N-2}{2}q}\int_{\Delta_{\sigma}}
   \frac{{\textstyle\sum\lambda_{i}x_{i}^2}}{(\eps^2+
   |x'|^2)^{\frac{N-2}{2}q}}  
   \,dx'\right)\\
   &&+\ O\left(\eps^{\frac{N-2}{2}q}\int_{\Delta_{\sigma}}
   \frac{|x'|^3}{(\eps^2+|x'|^2)^{\frac{N-2}{2}q}}  
   \,dx'\right)\\
   &&+\ O\lb\eps^{N-\frac{1}{\bt}}\rb\\
   &=&O\left(\eps^{\frac{1}{\bt}+1}\int_{\frac{\Delta_{\sigma}}{\eps}}
   \frac{|y'|^2}{(1+|y'|^2)^{\frac{N-2}{2}q}}  
   \,dy'\right)\\
   &&+\ O\left(\eps^{\frac{1}{\bt}+2}\int_{\frac{\Delta_{\sigma}}{\eps}}
   \frac{|y'|^3}{(1+|y'|^2)^{\frac{N-2}{2}q}}  
   \,dy'\right)\\
   &&+\ O\lb\eps^{N-\frac{1}{\bt}}\rb\\
   &=&O(\eps^{\frac{1}{\bt}+1}).
\eas

Combining the estimates for the three terms on the right hand side 
of (\ref{sumudt}),
$$
|U_{\eps}|_{q}^q=\frac{1}{2}B(q,N)\eps^\frac{1}{\bt}+
O\lb\eps^{\frac{1}{\bt}+1}\rb.
$$

\end{document}